\documentclass[11pt]{article}
\usepackage[dvips]{graphicx}
\usepackage[dvips,letterpaper,margin=1in]{geometry}
\usepackage{xcolor}
\usepackage{amsmath}
\usepackage{amsfonts}
\usepackage{amssymb}
\usepackage{amsthm}
\usepackage{newlfont}
\usepackage{epsfig}
\usepackage{multirow}
\usepackage{placeins}

\usepackage{enumerate}

\usepackage{nicefrac}

\usepackage[all,cmtip]{xy}
\definecolor{dgreen}{RGB}{49,128,23}
\definecolor{nicepink}{RGB}{255, 0, 102}
\definecolor{nicered}{RGB}{255, 80, 80}
\definecolor{goldy}{RGB}{230, 166, 2}

\usepackage{setspace}
\usepackage{hyperref}
\usepackage{cite}
\usepackage[capitalise,nameinlink]{cleveref}
\usepackage{tikz}
\usetikzlibrary{arrows.meta,automata,positioning}

\usepackage[format=plain,indention=.5cm, font={it,small},labelfont=bf]{caption}
\usepackage{subcaption}
\usepackage[normalem]{ulem}
\usepackage{bm}
\usepackage{soul}
\setstcolor{red}
\renewcommand{\b}{\bm}

\usepackage[numbers]{natbib}

\newcommand{\bR}{\mathbb{R}}

\newcommand{\cC}{\mathcal{C}}

\newcommand{\cE}{\mathcal{E}}

\newcommand{\cG}{\mathcal{G}}

\newcommand{\cK}{\mathcal{K}}

\newcommand{\cR}{\mathcal{R}}
\newcommand{\cS}{\mathcal{S}}

\newcommand{\cV}{\mathcal{V}}
\newcommand{\cW}{\mathcal{W}}

\newcommand{\ka}{k}

\newtheorem{definition}{Definition}[section]
\newtheorem{lemma}{Lemma}[section]

\newtheorem{theorem}{Theorem}[section]

\newtheorem{remark}{Remark}[section]

\newtheorem{example}{Example}[section]
\newtheorem{question}{Question}

\begin{document}
\small

\title{On classes of reaction networks and their associated polynomial dynamical systems}
\author{David F. Anderson\thanks{Department of Mathematics, University of Wisconsin-Madison, anderson@math.wisc.edu}, James D. Brunner\thanks{Division of Surgical Research, Department of Surgery, Mayo Clinic, brunner.james@mayo.edu}, Gheorghe Craciun\thanks{Departments of Mathematics and Biomolecular Chemistry, University of Wisconsin-Madison, craciun@wisc.edu}, and Matthew D. Johnston\thanks{Department of Mathematics \& Computer Science, Lawrence Technological University, mjohnsto1@ltu.edu}}
\date{\today}
\maketitle

\begin{abstract}
\small
In the study of reaction networks and the polynomial dynamical systems that they generate, special classes of networks with important properties have been identified. These include \emph{reversible}, \emph{weakly reversible}, and, more recently, \emph{endotactic} networks. While some inclusions between these network types are clear, such as the fact that all reversible networks are weakly reversible, other relationships are more complicated. Adding to this complexity is the possibility that inclusions be at the level of the dynamical systems generated by the networks rather than at the level of the networks themselves. We completely characterize the inclusions between reversible, weakly reversible, endotactic, and strongly endotactic network, as well as other less well studied network types. In particular, we show that every strongly endotactic network in two dimensions can be generated by an extremally weakly reversible network. We also introduce a new class of \emph{source-only} networks, which is a computationally convenient property for networks to have, and show how this class relates to the above mentioned network types.
\end{abstract}

\\


\bigskip

\section{Introduction}
\label{introduction}

\emph{Chemical reaction networks} model the behavior of sets of reactants, usually termed \emph{species}, that interact at specified rates to form sets of products. Under simplifying assumptions such as (i) a well-mixed reaction vessel, (ii) a sufficiently large number of reactants, and (iii) mass action kinetics, the dynamics of the concentrations of the species can be modeled by a system of autonomous polynomial ordinary differential equations. Such  systems are  known as \emph{mass action systems}.

With the increased recent interest in systems biology, significant attention has been given to the question of how dynamical properties of a mass action system can be inferred from the structure of the network of interactions that it models. In particular, it is of great interest to identify network structures that inform the dynamics regardless of the choice of parameters for the model (which are often unknown, or only known up to order of magnitude). This is not a trivial endeavor, as kinetic systems based on reaction networks are known to permit a wide variety of dynamical behaviors, including asymptotic convergence to a unique steady state \citep{H-J1}, multistationarity \citep{C-F1,C-F2}, periodicity and Hopf bifurcations \citep{W-H1,W-H2}, and chaotic behavior \citep{E-T}. Nevertheless, structural properties that have strong implications for the corresponding dynamical systems have been found.  See B.L. Clarke, \emph{Stability of complex reaction networks} \citep{clarke1980stability} for a thorough description of the classical problems and the connections between structural properties and stability of a network.

 In the papers \citep{F1,H,H-J1}, Feinberg, Horn, and Jackson introduced the now-classical notion of network deficiency and proved that weak reversibility and a deficiency of zero suffice for characterizing the steady state and local convergence properties of the corresponding mass action system. Moreover, their results hold regardless of the choice of parameters.  These papers are commonly credited as providing the framework for so-called \emph{chemical reaction network theory} \citep{Fe2,Fe4,Sh-F,F2,F3}. Chemical reaction network theory remains an active area of research to date, with a focus on both deterministic \citep{A,A4,ACKN2018, A2011bounded,A-S,yu2018,D-B-M-P,C-D-S-S,J-S7,B-P,Conradi19175,Hell2015DynamicalFO, craciun2019realizations} and stochastic models \citep{AK2018, AN2019, ACKN2018, ACKK2018, AY2019,A-C-K, AC2016, ACGW:lyapunov, dembo2018, CW:product, AM2018, AEJ2014}.  

We note that there are other methods in the literature that focus on the network structure of chemical interactions.  This includes work related to the discovery of new reactions \citep{herges1990reaction}, and on the understanding of the quantum physical and energetic properties of chemical reaction paths  \citep{mezey1982quantum,mezey1982topology,mezey1995reaction}. However, in the present work we focus on abstract mathematical models of chemical reaction networks defined as in the works of Feinberg, Horn, Jackson, and Clarke \citep{clarke1980stability,F1,H,H-J1} and do not consider mechanisms beyond the assumption of mass action.

Chemical reaction network theory can be applied either through direct knowledge or hypothesis of a network of interactions, or by constructing such a network from a system of ordinary differential equations with polynomial right hand sides. In either case, the network in question may have no special properties that can be used to draw conclusions. However, considering \emph{dynamical equivalence} may allow the application of a result that is not obviously relevant \citep{C-P,J-S2,Sz2,craciun2018}. Dynamical equivalence concerns the case of two distinct chemical reaction networks taken with mass action kinetics having identical governing systems of differential equations. It is simple to observe that not all dynamically equivalent network representations share the same structural properties. For example, consider the network
\begin{equation}\label{eq:9807070}
\xymatrix{
	\emptyset \ar[r]^{k_1} & X & \ar[l]_{k_2} 2X.
}
\end{equation}
This network is not weakly reversible (see \cref{classificationssection}). Under mass action kinetics, however, we may easily check that the network
\begin{equation}\label{eq:98}
\xymatrix{
	\emptyset \ar@<0.5ex>[r]^{\nicefrac{k_1}{2}} & \ar@<0.5ex>[l]^{\nicefrac{k_2}{2}} 2X
}
\end{equation}
generates exactly the same differential equation, $\dot{x} = k_2 - k_1x^2$, and therefore is a dynamically equivalent representation. This network, however, is reversible and therefore weakly reversible. Thus, existing theory may be used to immediately characterize the long-time behavior of the  dynamical system that is associated with both systems.

In the simple example introduced above, the notion of dynamical equivalence allowed us to make a conclusion about the long term behavior of a dynamical system associated with a network that did not appear to fit the hypothesis of the classical theorems of chemical reaction network theory. Furthermore, any model development and fitting from data must account for dynamical equivalence \citep{craciun2013statistical,C-P}. The notion of dynamical equivalence therefore plays an essential role in the study of mass action reaction networks. We therefore ask the following question:
\begin{question}\label{q1}
Are there easily checkable (geometric) conditions under which two networks are dynamically equivalent (in that they generate the same system of differential equations)?
\end{question}

A recent addition to the class of networks for which results can be obtained is the class of \emph{endotactic networks}, which include reversible and weakly reversible networks as subclasses.  Endotactic networks were first introduced  in Craciun et al. \citep{C-N-P} and, roughly speaking, a network is endotactic if the reactions of the network are ``inward-pointing'' in relation to the convex hull of the source nodes when the network is embedded in $\bR^d_{\geq 0}$.  See \cref{classificationsEG} for a precise formulation.  
 The deterministic dynamical systems corresponding to endotactic networks are conjectured to have positive solutions which are bounded and strictly positive for all time under  mild conditions on the reaction kinetics \citep{C-N-P}. This conjecture is known to be true in special cases, including when the network's stoichiometric subspace is two-dimensional or less \citep{Pa} and when the network satisfies an additional condition to make it \emph{strongly endotactic} \citep{ACKN2018, Gopal2014}. 

What is not known is exactly how endotactic networks fit into the hierarchy of well-studied network classifications such as reversible networks, weakly reversible networks, single linkage class networks, networks with a single terminal linkage class, and consistent networks. This is a non-trivial question in the context of mass action kinetics given recent work on dynamical equivalence. In fact there are many endotactic networks, including that shown in \cref{eq:9807070} for which we can find a dynamically equivalent weakly reversible network. 
We therefore ask the following question:
\begin{question}\label{q2}
Given the flexibility afforded by dynamical equivalence, how closely related are endotactic networks to the well-studied classifications of reversible, weakly reversible, and consistent networks?
\end{question}

Additionally, we introduce the notion of ``source-only networks" and show how endotactic and strongly endotactic networks relate to this class of networks. 

To answer both questions, we introduce a general framework in which to consider dynamical equivalence, including defining a reaction network as an embedded graph, called a Euclidean embedded graph (E-graph). This definition is equivalent to the classical definition found in the literature, notably Feinberg \citep{F3}, in the sense that it models the same dynamical system. 

In this paper, we show that, although significant overlaps exists, endotacticity is indeed distinct from weak reversibility. \cref{figure1}, \cref{figure2}, and \cref{figure3} give examples of endotactic and even strongly endotactic networks which cannot be realized as weakly reversible networks. We characterize overlap between these types of networks by analyzing the notion of ``dynamical equivalence" (\cref{dyneqdef}, \cref{overlap}, and \cref{includedyn}) under which distinct networks may give rise to the same dynamical systems. We also give checkable conditions for dynamical equivalence (\cref{containcondit}). We show that in two dimensions, strong endotacticity is equivalent to weak reversibility on an important subset of the nodes of the network (\cref{wkrev}), but \cref{figure3} gives a counterexample in three dimensions.

\cref{relationships} summarizes our results by giving a succinct summary of the relationships between classifications of networks. Moreover,  \cref{relationships} is complete in the sense that any additional paths would be false.

\section{Background}

In this section, we introduce background notation and results related to \emph{chemical reaction network theory} and \emph{mass action systems}, in particular.

\subsection{Chemical Reaction Networks}
Classically, a reaction network has been defined as below \citep{F3}:
\begin{definition}\label{SCR}
A \textbf{chemical reaction network} is a triple of finite sets $(\mathcal{S},\mathcal{C},\mathcal{R})$ where:
\begin{enumerate}
\item
The \textbf{species set} $\mathcal{S} = \{ X_1, \ldots, X_d \}$ consists of the basic species/reactants capable of undergoing chemical change.
\item
The \textbf{complex set} $\mathcal{C} = \{ C_1, \ldots, C_n \}$ consists of linear combinations of species of the form
\[C_i = \sum_{j=1}^d y_{ij} X_j, \; \; \; \; \; i=1,\ldots,n.\]
The constants $y_{ij} \in \mathbb{R}_{\geq 0}$ are called \emph{stoichiometric coefficients} and determine the multiplicity of each species within each complex. We define the complex support vectors $\bm{y}_i = (y_{i1},y_{i2},\ldots,y_{id})$ and assume that each complex is stoichiometrically distinct, i.e. $\bm{y}_i \not= \bm{y}_j$ for $i\not= j$. For simplicity, we will allow the support vector $\bm{y}_i$ to represent the complex $C_i$.
\item
The \textbf{reaction set} $\mathcal{R} = \{ R_1, \ldots, R_r \}$ consists of elementary reactions of the form
\[R_k: \; \; \; \bm{y}_{\rho(k)} \longrightarrow \bm{y}_{\rho'(k)}, \; \; \; k=1, \ldots, r\]
where $\rho(k)=i$ if $\bm{y}_i$ is the reactant complex of the $k^{th}$ reaction, and $\rho'(k) = j$ if $\bm{y}_j$ is the product complex of the $k^{th}$ reaction. We require that $\rho(k) \ne \rho'(k)$ for each $k = 1,\ldots, r$. Reactions may alternatively be represented as ordered pairs of complexes, e.g. $R_k = (\bm{y}_i,\bm{y}_j)$ if $\bm{y}_i \to \bm{y}_j$ is in the network.
\end{enumerate}
\end{definition}

We present the preceding classical definition (\cref{SCR}) in order to connect our results to the bulk of the literature in chemical reaction network theory. In some recent work (see \citep{rrobust,Cr,gheorgheToricDI,boros2019weakly, craciun2019endotactic}), chemical reaction networks have been defined in terms of a \emph{Euclidean Embedded Graph (E-graph)}. In this paper, we prefer to use this newer formulation, given below in \cref{egraphCRN}, due to its convenient geometric properties.

\begin{definition}
	A \emph{Euclidean embedded graph} (E-graph) $\cG = (\cV,\cE)$ is a finite directed graph whose nodes $\cV$ are distinct elements of a finite set $Y \subset \mathbb{R}^d$.
\end{definition}

It is convenient to define for each edge $e\in \cE$ a \emph{source vector} $\bm{s}(e)\in Y$, the label of the source node of $e$, the \emph{target vector} $\bm{t}(e)\in Y$, the label of the target node, and the \emph{reaction vector} $\bm{v}(e) = \bm{t}(e) - \bm{s}(e)$. We may regard $\bm{s}(e)$ as the source complex of some reaction while $\bm{t}(e)$ is the product complex of that same reaction.  

Now we define a chemical reaction network to simply be an E-graph for which a set of simple conditions hold.

\begin{definition}\label{egraphCRN}
	A \emph{reaction network} is a Euclidean embedded graph, $(\cV,\cE)$, whose nodes $\cV$ are labeled with distinct elements of a finite set $Y \subset \mathbb{R}^d_{\geq 0}$, and for which the following conditions hold:
	\begin{enumerate}
	\item $\cV \ne \emptyset$;
	\item for each $y \in \cV$ there exists $e\in \cE$ for which $\bm t(e) = y$ or $\bm s(e) = y$;
	\item $\bm{t}(e) \ne \bm{s}(e)$ for each $e \in \cE$.  That is, we never have $\b v(e) = \b 0$.
	\end{enumerate}
\end{definition}

 \cref{SCR} and \cref{egraphCRN} of a chemical reaction network are equivalent in the following sense: if we regard the set $\cS$ as the standard basis in $\bR^d$, then the set of vertices $\cV$ and edges $\cE$ in \cref{egraphCRN} can be chosen to be the set of complexes $\cC$ and reactions $\cR$ in \cref{SCR}.  It is most common to assume that $Y \subset \mathbb{Z}^d_{\ge 0}$.  Further, it is convenient to enumerate the elements of $\cE$, so that $\cE  = \{e_1,\dots, e_{|\cE|}\}$.

\subsection{Mass Action Systems}\label{mass_action_sec}

In this paper, we will focus on dynamical systems that are generated by reaction networks according to \emph{mass action kinetics} \citep{H-J1,F3}. We will denote the vector whose $i$th component gives the concentration of the $i$th species at time $t$ by  $\mathbf{x}(t)  \in \mathbb{R}_{\geq 0}^d$.  As is usual, we will often drop the $t$ in the notation and simply denote the concentration by $\mathbf{x}.$  Also, for two vectors $\b{u},\b{v} \in \mathbb{R}_{\ge 0}^d,$ we will denote 
\[
\b{u}^{\b{v}} = \prod_{i=1}^d u_i^{v_i}
\]
where we take $0^0=1$.
A system is  said to have mass action kinetics if the rate associated to reaction $i$ is 
\[
   k_i \mathbf{x}^{\bm{s}(e_i)}
\]
for some constant $k_i >0$,  called the \emph{rate constant} of the reaction. That is, the rate of each reaction is assumed to be proportional to the product of the concentrations of the constituent reactants, counted according to multiplicity. For example, a reaction of the form $X_1+ X_2 \to \cdots$ would have rate equal to $k \mathbf{x}_1 \mathbf{x}_2$ for some $k>0$, and a reaction of the form $X_1 + 2X_2 \to \cdots$ would have rate  equal to $\tilde k \mathbf{x}_1 \mathbf{x}_2^2$, for some $\tilde k > 0$.  Other common kinetic assumptions, especially in systems biology, are Michaelis-Menten kinetics \citep{M-M} and Hill kinetics \citep{Hi}. Since reaction $i$ pushes the system in the direction $\bm{v}(e_i)$, we have the following.

\begin{definition}\label{def:generates}
Given a reaction network $\cG = (\cV,\cE)$ as in \cref{egraphCRN} and, after enumerating $\cE$, a choice of rate constants $\cK = \{k_1,...,k_{|\cE|}\}\subset \bR_{>0}$, we say that $\cG$ \emph{generates} the  dynamical system $\cG(\cK)$ 
\begin{equation}\label{gener}
\frac{d\mathbf{x}}{dt} = \sum_{i = 1}^{|\cE|} k_{i} \mathbf{x}^{\bm{s}(e_i)}\bm{v}(e_i).
\end{equation}
\end{definition}

We will use the notation $\bm{f}_{\cG(\cK)}(\mathbf{x})$ to refer to the right hand side of the dynamical system in \cref{gener}. It is clear from \cref{gener} that every mass action system has the properties that $\frac{d\mathbf{x}}{dt} \in S = \mathit{Span}\{\bm{v}(e)|e\in \cE\}$. Consequently, solutions of \cref{gener} are restricted to \emph{stoichiometric compatibility classes} $(\mathbf{x}_0 + S) \cap \mathbb{R}_{\geq 0}^d$ \citep{V-H}.

As we noted in the introduction, different Euclidean embedded graphs (combined with  choices of rate constants) can generate the same polynomial dynamical system.

\subsection{Network Classifications}
\label{classificationssection}

A key feature of chemical reaction network theory is the attempt to relate dynamical properties of kinetic systems, and in particular mass action systems, to structural properties of the underlying reaction graphs. We therefore introduce the following foundational structural properties of chemical reaction networks.

\begin{definition}
	\label{classificationsEG}
	Consider a  reaction network $\cG = (\cV,\cE)$, where $\cE$ has been enumerated. The graph $\cG$ is said to be:
	\begin{enumerate}
		\item
		\textbf{consistent} if there is some choice of $a_1,a_2,...,a_{|\cE|} \in \mathbb{R}_{>0}$ such that $0 = \sum_{i=1}^{|\cE|} a_i \bm{v}(e_i)$. 
		\item
		\textbf{weakly reversible} if each connected component of the graph is strongly connected,  or, equivalently, each edge $e \in \cE$ is contained in a cycle.
		\item
		\textbf{endotactic} if, for every $\bm{w} \in \mathbb{R}^d$ and every $e_i \in \mathcal{E}$, $\bm{w} \cdot \bm{v}(e_i) < 0$ implies that there exists $e_j \in \cE$ such that 
		$\bm{w} \cdot (\bm{s}(e_j)-\bm{s}(e_i)) < 0$ and $\bm{w} \cdot \bm{v}(e_j) > 0$.
				\item 
		\textbf{strongly endotactic} if, for every $\bm{w} \in \mathbb{R}^d$ and every $e_i \in \mathcal{E}$, $\bm{w} \cdot \bm{v}(e_i) < 0$ implies that there exists $e_j\in\cE$ such that $\bm{w} \cdot (\bm{s}(e_j)-\bm{s}(e_i)) < 0$ and $\bm{w} \cdot \bm{v}(e_j) > 0$ and furthermore $\bm{w} \cdot (\bm{s}(e_j)-\bm{s}(e_k)) \leq 0$ for all $e_k \in \cE$.
	\end{enumerate}
\end{definition}

\noindent Consistency is closely related to a mass action system's capacity to admit positive steady states \citep{angeli2009tutorial}. Weak reversibility was introduced as a generalization of reversibility in Horn \& Jackson \citep{H-J1}. Endotactic networks were introduced as a generalization to weak reversibility in Craciun et al. \citep{C-N-P}, where it is shown that any weakly reversible network is endotactic.

\begin{remark}
Notice that consistency is a necessary but not sufficient condition on the network structure for the corresponding mass action system to admit positive steady states. For example, consider the network $\emptyset \longleftarrow X \longrightarrow 2X$.  This network is consistent, which  can be observed by selecting rate constants $a_1 = a_2 = 1$.  However, if the rate constants are selected as $\emptyset \; \stackrel{1}{\longleftarrow} \; X \; \stackrel{2}{\longrightarrow} \; 2X$, then the generated dynamics are $\dot{x} = x$, which has  solution $x(t) = x(0)e^t$,  and there is no positive steady state.  \hfill $\diamond$
\end{remark}

\begin{remark}
 Weak reversibility may also be understood using \cref{egraphCRN} as the property that every edge in $\cG$ is in a directed cycle. This is distinct, but similar, to the stronger requirement of \emph{reversibility}, i.e. that for every edge $e$ there is some edge $e^*$ such that $\bm{s}(e) = \bm{t}(e^*)$ and $\bm{t}(e) = \bm{s}(e^*)$. For instance, the network
\[
\xymatrix{
	X_2 \ar[r]\ar@<0.5ex>[dr]  & X_1 + X_2 \ar[d]\\
	 &\ar@<0.5ex>[ul] X_1}
 \]
is not reversible, since there is no reaction $X_1 + X_2 \to X_2$, but is weakly reversible since there is a path from $X_1+X_2$ to $X_2$ through $X_1$. \hfill $\diamond$
\end{remark}

	\begin{remark}\label{sweep}
		Intuitively, a Euclidean embedded graph is endotactic if no reaction ``points outward." This can be tested using the so-called ``parallel sweep test" (see Craciun et al. \citep{C-N-P}). In \cref{dyneqEgraphs} (a), we can tell that the network is endotactic because any direction $\b{w}$ which is not perpendicular to the two reactions has the property that if $\b{w}\cdot \b{v}(e_1) < 0$ then $\b{w} \cdot (\b{s}(e_2)-\b{s}(e_1)) = \b{w} \cdot \b{v}(e_1) < 0$ and $\b{w} \cdot \b{v}(e_2) = \b{w} \cdot (-\b{v}(e_1)) >0$. In \cref{dyneqEgraphs} (b), we can see from $\b{w} = (-1,-1)$ that this is not endotactic. We have that $\b{w} \cdot \b{v}(e_2) <0$, but $\b{w} \cdot (\b{s}(e_1) - \b{s}(e_2)) = \b{w} \cdot \b{v}(e_1) = 0$, and while $\b{w} \cdot \b{v}(e_3)>0$, $\b{w}\cdot(\b{s}(e_2) -\b{s}(e_3)) = 0$. 
\hfill $\diamond$
	\end{remark}
	
	\begin{figure}
	\centering
	\includegraphics[scale = 1]{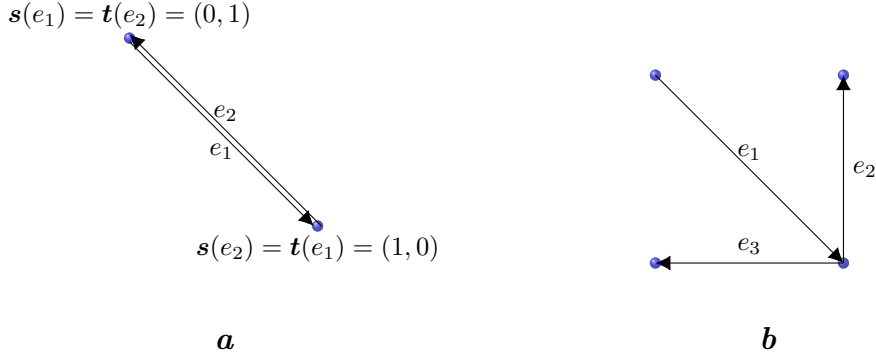}
\caption{Two systems which have the same dynamics when every rate constant is taken to be $1$. Notice that $\bm{v}(e_2)$ in (a) is in the positive cone formed by $\bm{v}(e_2)$ and $\bm{v}(e_3)$ in (b), and that all of these edges have the same source vector, $(1,0)$.
}\label{dyneqEgraphs}
\end{figure}

The properties above are intrinsically properties of E-graphs. However, we can define the same notions for dynamical systems using the relationship between E-graphs and dynamical systems established in \cref{def:generates}.

\begin{definition}\label{dyneq2}
	We will call a  dynamical system $\frac{d\mathbf{x}}{dt} = \b{f}(\mathbf{x})$:
	\begin{enumerate}
		\item
		\textbf{consistent} if $\frac{d\mathbf{x}}{dt} =\b{f}(\mathbf{x})$ can only be generated by a consistent Euclidean embedded graph;
		\item
		\textbf{weakly reversible} if $\frac{d\mathbf{x}}{dt} = \b{f}(\mathbf{x})$ can be generated by some weakly reversible Euclidean embedded graph $\cG$;
		\item
		\textbf{endotactic} if $ \frac{d\mathbf{x}}{dt} =\b{f}(\mathbf{x})$ can be generated by some endotactic Euclidean embedded graph $\cG$;
		\item 
		\textbf{strongly endotactic} if $\frac{d\mathbf{x}}{dt} =\b{f}(\mathbf{x})$ can be generated by some strongly endotactic Euclidean embedded graph $\cG$.
	\end{enumerate}
\end{definition}

\begin{remark}\label{consistent}
	In \cref{dyneq2} part 1., we have that if a dynamical system is consistent, then every graph which generates the system must be consistent. This is in contrast to parts 2., 3. and 4. of this definition. This is because any polynomial dynamical system can be generated by some consistent network. To see this, we simply add a set of edges $e_1^*,...,e_p^*$ which share a source $\b{s}^*$ to some graph $\cG = (\cV,\cE)$ which generates the system, requiring that the cone generated by $\b{v}(e_1^*),...,\b{v}(e_p^*)$ is equal to the span of the original reaction vectors. This implies that $0\in \mathit{Cone}(\{\b{v}(e_1^*) ,...,\b{v}(e_p^*)\})$ and so the new graph also generates the polynomial. Also, for any choice of $a_1,...,a_{|\cE|} > 0$, we have chosen the $\b{v}(e^*_i)$ such that $-\sum_{i=1}^{|\cE|}a_i \b{v}(e_i)\in \mathit{Cone}(\{\b{v}(e_1^*) ,...,\b{v}(e_p^*)\})$, and furthermore is in the relative interior of that cone. Therefore, there is a choice of $b_1,...,b_p > 0$ such that $\sum_{i=1}^p b_i\b{v}(e_i^*) = -\sum_{i=1}^{|\cE|}a_i \b{v}(e_i)$ and so the new graph is consistent.
	\hfill $\diamond$
\end{remark}

\subsection{Dynamical Equivalence}
\label{desection}

As already noted, it is well known that the dynamical representations of mass action systems (i.e., \cref{gener}) are not uniquely determined by the network structure. For another example, consider the following networks, whose E-graphs are shown in \cref{dyneqEgraphs}:
\begin{equation}
\label{system1}
\xymatrix{
	X_1 \ar@<0.5ex>[r]^{1} & \ar@<0.5ex>[l]^{1} X_2
}
\end{equation}
and
\begin{equation}
\label{system2}
\xymatrix@!0@R=6mm@C=10mm{
	&  &**[r] X_1 + X_2\\
	X_2 \ar[r]^1 & X_1  \ar[ur]^1 \ar[dr]_1& \\
	 &  &**[r] \emptyset
}
\end{equation}
It can be easily seen that \cref{system1} and \cref{system2} are both governed by the mass action dynamics $\dot{x}_1 = -\dot{x}_2 = -x_1+x_2$.  We therefore introduce the following definition \citep{C-P}.

\begin{definition}\label{dyneqdef}
Consider two chemical reaction networks $\cG = (\cV,\cE)$ and $\tilde{\cG} = (\tilde{\cV},\tilde{\cE})$, combined with rate constants $\mathcal{K} = \{ k_i \; | \; i=1,\ldots, r\}$ and $\tilde{\mathcal{K}} = \{ \tilde{k}_i \; | \; i=1,\ldots, \tilde{r}\}$, respectively. We will say that the mass action systems $\cG(\mathcal{K})$ and $\tilde{\cG}(\tilde{\mathcal{K}})$ are \textbf{dynamically equivalent} if the generated functions $\bm{f}_{\cG(\cK)}$ and $\bm{f}_{\tilde{\cG}(\tilde{\cK})}$ coincide (i.e.  $\bm{f}_{\cG(\cK)}(\mathbf{x}) = \bm{f}_{\tilde{\cG}(\tilde{\cK})}(\mathbf{x})$, for all $\mathbf{x}$).
\end{definition}

We can see that the dynamical systems   $\mathcal{G}(\mathcal{K})$ and $\tilde{\mathcal{G}}(\tilde{\mathcal{K}})$ associated with \cref{system1} and \cref{system2} are dynamically equivalent. We may furthermore observe that $\cG$ and $\tilde{\cG}$ fail to share the same structural properties:  $\cG$ is weakly reversible while $\tilde{\cG}$ is not. As in \cref{sweep}, we can also easily verify that $\cG$ is endotactic, while $\tilde{\cG}$ is not  (see \cref{dyneqEgraphs}). Notice, however, that in the classifications given in \cref{dyneq2}, the polynomial dynamical system \textit{is} said to be weakly reversible (and endotactic) because the network $\cG = (\cV,\cE)$ is.  This example shows that it is possible for a mass action system to behave as though the generating network has a particular desirable network property, even when the generating network does not itself have it. We might therefore say that the known generating E-graph behaves as though it has that property for some (or perhaps even all) choices of rate constants.

In order to formalize this notion, we inspect the case of E-graphs which may generate the same polynomial dynamical system.

\begin{definition}\label{overlap}
	Let $\cG_1$ and $\cG_2$ be Euclidean embedded graphs. We say that $\cG_1$ and $\cG_2$ {\bf have the capacity for dynamical equivalence}, and write $\cG_1 \sqcap \cG_2$, if there exists a system $\frac{d\mathbf{x}}{dt} = \bm{f}(\mathbf{x})$ that can be generated by both $\cG_1$ and $\cG_2$ (i.e. there exists $\cK_1$ and $\cK_2$ such that $\bm{f}_{\cG_1(\cK_1)}(\mathbf{x}) = \bm{f}_{\cG_2(\cK_2)}(\mathbf{x}) = \bm{f}(\mathbf{x})$, for all $\mathbf{x}$).
\end{definition}

\begin{definition}
	Let $\cG$ be a Euclidean embedded graph. We say that $\cG$ {\bf has the capacity for weak reversibility} if there exists a weakly reversible Euclidean embedded graph $\tilde{\cG}$ such that $\cG \sqcap \tilde{\cG}$. Likewise, we say that $\cG$ {\bf has the capacity to be endotactic (strongly endotactic)} if there exists some endotactic (strongly endotactic) Euclidean embedded graph $\tilde{\cG}$ such that $\cG \sqcap \tilde{\cG}$.
\end{definition}

The following theorem asserts that the above definitions are meaningful. In particular, it shows that there are networks which can generate weakly reversible (respectively, endotactic) networks which are not themselves weakly reversible (respectively, endotactic).

\begin{theorem}\label{inclusions}
	The following inclusions hold and are strict.
	\begin{enumerate}
		\item
		The set of networks with the capacity for weak reversibility contains the set of weakly reversible networks.
		\item
		The set of networks with the capacity to be endotactic contains the set of endotactic networks.
	\end{enumerate}
	\end{theorem}
	\begin{proof} That these inclusions hold follows directly from the definitions.  We now simply need to demonstrate that the reverse inclusion does not hold.
		
		Consider the network $\cG$, below
		\begin{equation}\label{gbiggerone}
		\xymatrix{
			X_2 \ar[r]^{k_3} \ar[d]_{k_4} & X_1 + X_2 \\
			\emptyset & \ar[l]_{k_2} \ar[u]_{k_1} X_1
		}.
		\end{equation}
		The network $\mathcal{G}$ is  neither weakly reversible nor endotactic. However,  if (and only if) $k_1=k_2$ and $k_3=k_4$, then the generated mass action system may also be generated by the weakly reversible and endotactic network $\tilde{\cG}$ with correctly chosen rate constants:
		\begin{equation}\label{gsmallerone}
		\xymatrix{
			X_1 \ar@<0.5ex>[r]^{{k}_1} & \ar@<0.5ex>[l]^{{k}_3} X_2
		}.
		\end{equation}
		Therefore, $\cG \sqcap \tilde{\cG}$, completing the proof. 
	\end{proof}

Notice that we may also define the notion of \emph{having the capacity to be consistent} in the same way. However, by \cref{consistent} we see that every E-graph has the capacity to be consistent.

We are interested in the stronger case in which every system generated by an E-graph can also be generated by another E-graph with a desired property. We therefore define the following.

\begin{definition}
	Let $\cG$ be a Euclidean embedded graph. We say that $\cG$ is {\bf effectively weakly reversible} if every  dynamical system generated by $\cG$ is weakly reversible. Likewise, we say that $\cG$ is {\bf effectively endotactic (strongly endotactic)} if every polynomial dynamical system generated by $\cG$ is endotactic (strongly endotactic).
\end{definition}

It is possible for a network to be effectively weakly reversible, but not weakly reversible as the example \cref{eq:9807070}, given in the introduction, demonstrates.

In order to show that a network is effectively weakly reversible or effectively endotactic, it is of course helpful to characterize when one network can generate any system that can be generated by some other system.

\begin{definition}\label{includedyn}
	Let $\cG_1$ and $\cG_2$ be Euclidean embedded graphs. We say that $\cG_1$ includes the dynamics of $\cG_2$, and write $\cG_2 \sqsubseteq \cG_1$, if any system $\frac{d\mathbf{x}}{dt} = \bm{f}(\mathbf{x})$ generated by $\cG_2$ can also be generated by $\cG_1$ (i.e., for any $\cK_2$ there is some $\cK_1$ such that $\b{f}_{\cG_1(\cK_1)}(\mathbf{x}) = \b{f}_{\cG_2(\cK_2)}(\mathbf{x})$ for all $\mathbf{x}$). 
\end{definition}

We have already encountered  several examples of \cref{includedyn} in this manuscript. For example, the network shown in \cref{system2} and \cref{dyneqEgraphs} (b) contains the dynamics of the network shown in \cref{system1} and \cref{dyneqEgraphs} (a). In \cref{dyneqEgraphs}, it is noted that these networks generate the same dynamical system when every rate constant is taken to be $1$. Now, note that for any choice of rate constants $\ka_{1}$ for edge $e_1$ and $\ka_{2}$ for edge $e_2$ chosen to generate a dynamical system using the network shown in \cref{dyneqEgraphs} (a), we can generate the same network using \cref{dyneqEgraphs} (b) by choosing $\tilde{\ka}_1 = \ka_1$ for edge $e_1$, and $\tilde{\ka}_2 = \tilde{\ka}_3 = \ka_2$  for edges $e_2$ and $e_3$.  

An obvious sufficient condition for a network $\cG$ to be effectively weakly reversible is then that there exists some weakly reversible $\tilde{\cG}$ such that $\cG \sqsubseteq \tilde{\cG}$.

\section{Main Results}

In this section, we prove the main correspondences of this paper. In the first subsection, we address \cref{q1}. We then apply this result in the subsequent subsections in order to investigate \cref{q2}.

We will frequently require the following known results \citep{mangasarian1994nonlinear}.

\begin{lemma}[Farkas' Lemma]
\label{lemma01}
Let $\{ \bm{v}_i \}$, $i=1, \ldots, m$, denote a family of vectors in $\mathbb{R}^n$. Then, for any $\bm{b} \in \mathbb{R}^n$ exactly one of the following is true:
\begin{enumerate}
\item
There exist constants $\lambda_i \geq 0$, $i=1, \ldots, m$, such that $\displaystyle{\bm{b} = \sum_{i=1}^m \lambda_i \bm{v}_i}$; or
\item
There is a vector $\bm{w} \in \mathbb{R}^n$ such that $\bm{w} \cdot \bm{v}_i \geq 0$ for $i=1,\ldots, m$, and $\bm{w} \cdot \bm{b} < 0.$
\end{enumerate}
\end{lemma}

\begin{lemma}[Stiemke's Theorem]
\label{stiemke}
Let $\{ \bm{v}_i \}$, $i=1, \ldots, m$, denote a family of vectors in $\mathbb{R}^n$. Then exactly one of the following is true:
\begin{enumerate}
\item
There exist constants $\lambda_i \in \mathbb{R}_{>0}$, $i=1, \ldots, m$, so that $\displaystyle{\sum_{i=1}^m \lambda_i \bm{v}_i = \bm{0}}$; or
\item
There exists a vector $\bm{w} \in \mathbb{R}^n$ so that $\displaystyle{\bm{w} \cdot \bm{v}_i \leq \bm{0}}$, $i=1,\ldots, m$, with the inequality strict for at least one $i_0 \in \{ 1, \ldots, m\}$.
\end{enumerate}
\end{lemma}

\subsection{A condition for dynamical equivalence}

Here, we provide necessary and sufficient conditions under which one network contains the dynamics of another. Furthermore, as a corollary we provide necessary and sufficient conditions under which two networks have the capacity for dynamical equivalence. These conditions are geometric in nature and easily checkable.

The appearance of the source complexes $\bm{s}(e)$ as exponents in \cref{gener} suggests that we need to consider this subset of the complexes. We will also need the notion of a \emph{cone}. Given a finite set of vectors $S \subseteq \bR^d$ we define the set $K = \mathit{Cone}(S)$, the cone \emph{generated} by $S$, as the closed, convex set of all finite, nonnegative linear combinations of the elements of $S$  \citep{nonneg}. We denote the interior of a cone $K = \mathit{Cone}(S)$ relative to the span of $S$ (i.e. the relative interior of $K$) by $\mathit{RelInt}({K})$.

If $\cG = (\cV,\cE)$ is a Euclidean embedded graph, let $\cS\cC_{\cG} = \{\bm{s}(e)|e\in \cE\}$ be the source complexes/vectors of $\cG$, and for $\bm{s}\in \cS\cC_{\cG}$ let  $V^{\cG}(\bm{s}) = \mathit{Cone}(\{\bm{v}(e_i)|\bm{s}(e_i) = \bm{s}, e_i \in \cE\})$ be the cone generated by those reaction vectors with source vector equal to $\bm{s}$. If a vector $\bm{s} \not \in \cS\cC_{\cG}$, we  define $V^{\cG}(\bm{s}) = \{\bm{0}\}$. 

\begin{theorem}\label{containcondit}
	$\cG_2 \sqsubseteq \cG_1$ if and only if (i)  $\cS\cC_{\cG_2} \subset \cS\cC_{\cG_1}$, and (ii) for every $\bm{s} \in \cS\cC_{\cG_1}$
	\[
	\mathit{RelInt}(V^{\cG_2}(\bm{s})) \subseteq \mathit{RelInt}(V^{\cG_1}(\bm{s})),
	\]
	where we take $V^{\cG}(\bm{s}) = \{\bm{0}\}$ if $\bm{s} \not \in \cS\cC_{\cG}$.
\end{theorem}

\begin{proof}
	We begin by proving that (i) and (ii) imply that $\cG_2 \sqsubseteq \cG_1$.  Let $\cK^2 \in \bR^{|\cE_{\cG_2}|}_{>0}$. We must show that there exists $\cK^1 \in \bR^{|\cE_{\cG_1}|}_{>0}$ such that 
	\[
	\bm{f}_{\cG_2(\cK^2)}  = \bm{f}_{\cG_1(\cK^1)}.
	\]
	We use superscript $1$ or $2$ to differentiate edges, source vectors, and reaction vectors of $\cG_1$ and $\cG_2$, respectively. Note that for any choice of $\cK^1$ and $\cK^2$ we have
	\[
	\bm{f}_{\cG_2(\cK^2)}(\mathbf{x})  - \bm{f}_{\cG_1(\cK^1)}(\mathbf{x}) = \sum_{e^2 \in \cE_{\cG_2}} k_{e^2} \mathbf{x}^{\bm{s}(e^2)}\bm{v}(e^2) - \sum_{e^1 \in \cE_{\cG_1}} k_{e^1} \mathbf{x}^{\bm{s}(e^1)}\bm{v}(e^1).
	\]
	We wish to rewrite these sums in terms of the source complexes, which we enumerate via 
	\[
		\cS\cC_{\cG_1} \cup \cS\cC_{\cG_2} =  \cS\cC_{\cG_1} =  \{s_1,s_2,\dots,s_{|\cS\cC_{\cG_1}|}\}.
		\]
		Then, for each $\b{s}_i \in \cS\cC_{\cG_1}$ we let 
	 $m_i = |\{e^1 \in \cE_{\cG_1}| \bm{s}(e^1) = \bm{s}_i\}|$, and $n_i = |\{e^2 \in \cE_{\cG_2}| \bm{s}(e^2) = \bm{s}_i\}|$ be the number of edges out of complex $\b{s}_i$ for networks $\cG_1$ and $\cG_2$, respectively. Furthermore, let $\{k_{ij}^l\} = \{k_{e^l} | \bm{s}(e^l) = \bm{s}_i\}$, and $\{\bm{v}_{ij}^l\} = \{\bm{v}(e^l) | \bm{s}(e^l) = \bm{s}_i\}$, $l = 1,2$, where the sizes of the sets are $m_i$ and $n_i$ for $\ell = 1$ and $\ell = 2$, respectively. Then 
	\[
	\bm{f}_{\cG_2(\cK^2)}(\mathbf{x})  - \bm{f}_{\cG_1(\cK^1)}(\mathbf{x}) = \sum_{\bm{s}_i \in \cS\cC_{\cG_2}}\mathbf{x}^{\bm{s}_i}\sum_{j=1}^{n_i}k_{ij}^2 \bm{v}_{ij}^2 - \sum_{\bm{s}_i \in \cS\cC_{\cG_1}}\mathbf{x}^{\bm{s}_i}\sum_{j=1}^{m_i}k_{ij}^1\bm{v}_{ij}^1.
	\]
	Then, because $\cS\cC_{\cG_2} \subseteq \cS\cC_{\cG_1}$, 
	\[
	\bm{f}_{\cG_2(\cK^2)}(\mathbf{x})  - \bm{f}_{\cG_1(\cK^1)}(\mathbf{x}) = \sum_{\bm{s}_i \in \cS\cC_{\cG_2}}\mathbf{x}^{\bm{s}_i}\left(\sum_{j=1}^{n_i}k_{ij}^2 \bm{v}_{ij}^2 - \sum_{j=1}^{m_i}k_{ij}^1 \bm{v}_{ij}^1\right) - \sum_{\bm{s}_i \in \cS\cC_{\cG_1}\setminus \cS\cC_{\cG_2}}\mathbf{x}^{\bm{s}_i}\sum_{j=1}^{m_i}k^1_{ij}\bm{v}_{ij}^1.
	\]
	We have that $\b{0} \in \mathit{RelInt}(V^{\cG_1}(\bm{s}_i))$ if $\bm{s}_i \in \cS\cC_{\cG_1}\setminus \cS\cC_{\cG_2}$, so we can choose $\cK^1$ with $k^1_{ij} >0$ so that each term of the second sum is $\b{0}$. Furthermore, for each $\bm{s}_i\in \cS\cC_{\cG_2}$,
	\[
	\sum_{j=1}^{n_i}k_{ij}^2 \bm{v}_{ij}^2 - \sum_{j=1}^{m_i}k_{ij}^1\bm{v}_{ij}^1 = \bm{w} - \sum_{j=1}^{m_i}k_{ij}^1 \bm{v}_{ij}^1,
	\]
	where $\bm{w} \in \mathit{RelInt}(V^{\cG_2}(\bm{s}_i))$. Because of condition (ii) we may also conclude that $\bm{w} \in \mathit{RelInt}(V^{\cG_1}(\bm{s}_i))$.  Hence, we can choose $k_{ij}^1>0$ so that
	\[
	\bm{w} - \sum_{j=1}^{m_i}k_{ij}^1\bm{v}_{ij}^1 = 0.
	\]
	With these choice of parameters we have $\bm{f}_{\cG_2(\cK^2)}  = \bm{f}_{\cG_1(\cK^1)}$. 
	
	Next, we show that $\cG_2 \sqsubseteq \cG_1$ imply (i) and (ii) hold.  First, if there is some source $\bm{s}_i \in \cS\cC_{\cG_2}$ but $\bm{s}_i \not \in \cS\cC_{\cG_1}$, then for general choice $\cK^2$ there will be a monomial $\mathbf{x}^{\bm{s}_i}$ in $\bm{f}_{\cG_2(\cK^2)}$ which cannot appear in $\bm{f}_{\cG_1(\cK^1)}$. Therefore, (i) must hold.  
	
	Next, in order to find a contradiction suppose, that for some $\bm{s}_i \in \cS\cC_{\cG_1}$ there is a vector $\b{w} \in \mathit{RelInt}(V^{\cG_2}(\bm{s}_i))$ but $\bm{w} \not \in \mathit{RelInt}(V^{\cG_1}(\bm{s}_i))$, so that (ii) does not hold. Then, since $\b{w} \in \mathit{RelInt}(V^{\cG_2}(\bm{s}_i))$, there is a choice $\cK^2$ so that  $\bm{f}_{\cG_2(\cK^2)}$ contains a term of the form
	\[
	\mathbf{x}^{\bm{s}_i}\sum_{j=1}^{n_i}k_{ij}^2\bm{v}_{ij}^2 = \mathbf{x}^{\bm{s}_i}\bm{w}.
	\]
	Then, 	
	for any choice of $\cK^1 \in \bR^{|\cE_{\cG_1}|}_{>0}$ the function $\bm{f}_{\cG_2(\cK^2)}(\mathbf{x})  - \bm{f}_{\cG_1(\cK^1)}(\mathbf{x})$ will include the term
	\[
	\left(\bm{w} - \sum_{j=1}^{m_i}k_{ij}^1\bm{v}_{ij}^1\right)\mathbf{x}^{\bm{s}_i}.
	\]
	Because   $\bm{w} \not \in \mathit{RelInt}(V^{\cG_1}(\bm{s}_i))$,
	this term is non-zero for any $\mathbf{x} \neq 0$.  Moreover, no other term could cancel it since no other monomials of the form $\mathbf{x}^{\bm{s}_i}$ can appear.  Thus, $\bm{f}_{\cG_2(\cK^2)}(\mathbf{x})  \ne \bm{f}_{\cG_1(\cK^1)}(\mathbf{x})$.  Since this conclusion holds for \textit{any} choice of $\cK^1$, we  conclude that  we do not have $\cG_2 \sqsubseteq \cG_1$, which is a contradiction. 
\end{proof}

\begin{remark}\label{splitting}
	For any edge $e$ of a network $\cG$, we may generate a new network $\cG_e$ by removing $e$ and adding edges $e_1,...,e_d$ such that $\b{s}(e_1) =\cdots =\b{s}(e_d) = \b{s}(e)$ and $\b{v}(e) \in \mathit{RelInt}(\mathit{Cone}(\{\b{v}(e_1),...,\b{v}(e_d)\}))$. Then, \cref{containcondit} implies that $\cG_e$ contains the dynamics of $\cG$.  We often refer to this as ``splitting" the reaction vector $\b{v}(e)$. \hfill $\diamond$
\end{remark}

\cref{containcondit} shows when the dynamical systems generated from one network are contained within those generated by another.  It is reminiscent of Theorem 4.4 of \citep{C-P}, which gives conditions for when the two sets of dynamical systems intersect (i.e., when they have the capacity for dynamical equivalence).  Moreover, their proofs are similar. However, a missing case in the proof of Theorem 4.4 of \citep{C-P} was noted by G\'abor Szederk\'enyi in \citep{sz2009}.    A complete statement and proof of Theorem 4.4 of \citep{C-P} appear below.

\begin{theorem}[Complete version of Theorem 4.4 of \citep{C-P}]
	Under the mass-action kinetics assumption, two chemical reaction networks represented by the graphs $\cG_1$ and $\cG_2$ have the capacity for dynamical equivalence, i.e., $\cG_1 \sqcap  \cG_2$, if and only if for every $\bm{s} \in \cS\cC_{\cG_1} \cup \cS\cC_{\cG_2}$
	\[
	\mathit{RelInt}(V^{\cG_2}(\bm{s})) \cap \mathit{RelInt}(V^{\cG_1}(\bm{s})) \neq \emptyset
	\]
	where we take $V^{\cG}(\bm{s}) = \bm{0}$ if $\bm{s} \not \in \cS\cC_{\cG}$.
\end{theorem}

\begin{proof}

	We begin by showing that the above conditions imply that $\cG_1$ and $\cG_2$ have the capacity for dynamical equivalence.  
	We enumerate the set  $\cS\cC_{\cG_1} \cup \cS\cC_{\cG_2}$, and for each $\bm{s}_i$ in that set, we choose 
	\[
		\bm{w}_i \in \mathit{RelInt}(V^{\cG_2}(\bm{s})) \cap \mathit{RelInt}(V^{\cG_1}(\bm{s})).
	\]
	Let
	\[
		f(\mathbf{x}) = \sum_i \mathbf{x}^{\bm{s}_i} \bm{w}_i.
	\]
	Then both $\cG_1$ and $\cG_2$ generate $f$ and we conclude that $\cG_1 \sqcap  \cG_2$,.

	Next suppose that $\cG_1$ and $\cG_2$ have the capacity for dynamical equivalence. Hence, by
	 \cref{overlap}, there exists some polynomial
	\begin{equation}\label{adkjfa;jf}
	f(\mathbf{x}) = \sum_{i = 1}^n k_i \mathbf{x}^{\bm{s}_i}\b{w}_i,
	\end{equation}
	with $\bm{s}_i \ne \bm{s}_j$ for $i \ne j$, 
	that can be generated by both $\cG_1$ and $\cG_2$.  Let $\mathcal{T} = \{\b{s}_1,\dots, \b{s}_n\}$ and note that $\mathcal{T} \subset  \cS\cC_{\cG_1} \cup \cS\cC_{\cG_2}$.  For any $\b{s} \in ( \cS\cC_{\cG_1} \cup \cS\cC_{\cG_2}) \setminus \mathcal{T}$, implying no term in the polynomial corresponds with $\b{s}$, we immediately have that
		\[
		\b{0} \in \mathit{RelInt}(V^{\cG_2}(\bm{s})) \cap \mathit{RelInt}(V^{\cG_1}(\bm{s})), 
	\]
	as desired.  Furthermore, for each $\b{s}_i \in \mathcal{T}$,  we have 
	\[
		\emptyset \ne  \{\alpha \b{w}_i : \alpha > 0\} \subseteq  \mathit{RelInt}(V^{\cG_2}(\bm{s_i})) \cap \mathit{RelInt}(V^{\cG_1}(\bm{s_i})) 
		\]
		where containment follows by the fact that each of $\cG_1$ and $\cG_2$ generated $f$ in \cref{adkjfa;jf}.  Thus, we have shown the condition holds for all $\b{s} \in \cS\cC_{\cG_1} \cup \cS\cC_{\cG_2}$ and the proof is complete.
\end{proof}

\subsection{Endotactic Networks}

We turn to our study of endotactic and strongly endotactic networks.

\begin{lemma}\label{effectively_endo_endo}
	Let $\cG= (\cV,\cE)$ and $\tilde{\cG}= (\tilde{\cV},\tilde{\cE})$ be  reaction networks such that $\cG \sqsubseteq \tilde{\cG}$. If $\tilde{\cG}$ is  endotactic, then $\cG$ is  endotactic.  Moreover, if $\tilde{\cG}$ is strongly endotactic, then $\cG$ is strongly  endotactic.
\end{lemma}

\begin{proof}
	 Let $e_i \in \cE$ and $\b{w}$ be such that $\b{w}\cdot \bm{v}(e_i) < 0$.  We must show that  there exists $e_j \in \cE$ such that  $\bm{w} \cdot (\bm{s}(e_j)-\bm{s}(e_i)) < 0$ and $\bm{w} \cdot \bm{v}(e_j) > 0$.

	Because $\cG \sqsubseteq \tilde{\cG}$, we know from \cref{containcondit} that   $\bm{s}(e_i) \in \cS\cC_{\tilde{\cG}}$, and $\bm{v}(e_i) \in V^{\tilde{\cG}}(\bm{s}(e_i))$. There then exists $\tilde{e}_i\in \tilde{\cE}$ such that $\bm{s}(\tilde{e}_i) = \bm{s}(e_i)$ and $\bm{v}(\tilde{e}_i) \cdot \b{w}<0$, since $\b{w}\cdot \bm{v}(e_i) < 0$ 
\[
	\bm{v}(e_i) \in V^{\tilde{\cG}}(\bm{s}(e_i))\quad \Rightarrow \quad \b{v}(e_i) = \sum_{\tilde{e}_i|\bm{s}(\tilde{e}_i) = \bm{s}(e_i)} \lambda_i \b{v}(\tilde{e}_i)\quad \lambda_i \geq 0
\]
Because $\tilde{\cG}$ is endotactic, there is some $\tilde{e}_j$ such that $\b{w}\cdot (\bm{s}(\tilde{e}_j)-\bm{s}(\tilde{e}_i))<0$ and $\bm{v}(\tilde{e}_j)\cdot \b{w} > 0$.  Moreover, we can choose $\tilde e_j$ so that $\b w\cdot (\b s(\tilde e_j) - \b s(\tilde e_i))$ is minimal over edges satisfying  $\bm{v}(\tilde{e}_j)\cdot \b{w} > 0$. 
	
We can complete the proof by proving two statements.
	\begin{enumerate}[(1)]
	\item $\b s(\tilde e_j) \in \cS\cC_{\cG}$ (the source complexes for $\cG$).  
	\item If $\b u \in \mathit{RelInt}(V^{{\cG}}(\bm{s}(\tilde{e}_j))$, then $\b u \cdot \b w > 0$.
	\end{enumerate}
	Combining the above gives the existence of the necessary edge in $\cE$.
	
We prove (1) by showing that $\b{0} \not \in \mathit{RelInt}(V^{\tilde{G}}(\b{s}(\tilde{e}_j))$. Then, combined with \cref{containcondit}, $\b s(\tilde e_j) \in \cS\cC_{\cG}$. Suppose that 	$\b{0} \in  \mathit{RelInt}(V^{\tilde{G}}(\b{s}(\tilde{e}_j))$. Then, by \cref{stiemke}, there is some $\tilde{e}_k \in \tilde{\cE}$ such that $\b{s}(\tilde{e}_k) = \b{s}(\tilde{e}_j)$, and $\b{v}(\tilde{e}_k) \cdot \b{w} <0$. However, the minimality of $\b w\cdot (\b s(\tilde e_j) - \b s(\tilde e_i))$ over edges satisfying  $\bm{v}(\tilde{e}_j)\cdot \b{w} > 0$ then implies that there is no edge $\tilde{e}_l$ with $\b{v}(\tilde{e}_l)\cdot \b{w} > 0$ and $\b{w}\cdot(\b{s}(\tilde{e}_l)\b{w}(\tilde{e}_k)) < 0$. This contradicts the condition that $\tilde{\cG}$ is endotactic.
	
	To prove (2), let $\b u \in \mathit{RelInt}(V^{\tilde{\cG}}(\bm{s}(\tilde{e}_j))$.  Then, there are $\lambda_k>0$ and $\tilde e_k\in \tilde \cG$, each with source $\b s(\tilde e_j)$,  for which
	\[
	\b u = \sum_{k} \lambda_k \b v(\tilde e_k) =   \lambda_j \b v(\tilde e_j)  + \sum_{k\ne j} \lambda_k \b v(\tilde e_k).
	\]
	Dotting with $\b w$  yields
	\[
		\b u \cdot \b w = \lambda_j \b v(\tilde e_j) \cdot \b w + \sum_{k\ne j} \lambda_{k} \b v(\tilde e_k) \cdot \b w.
	\]
	If $\b u \cdot \b w  \le 0$, then, because $\b v (\tilde e_j ) \cdot \b w > 0$,  we would be forced to conclude $\b v(\tilde e_i) \cdot \b w < 0$ for some $i \ne j$.  Since $\tilde \cG$ is endotactic, this would contradict the minimality of $\b w\cdot (\b s(\tilde e_j) - \b s(\tilde e_i))$.  Hence, we can conclude that $\b u \cdot \b w  >0$. \cref{containcondit} implies that there is such a $\b{u} \in \mathit{RelInt}(V^{\cG})(\b{s}(\tilde{e}_j))$.

	Turning to the case of $\tilde \cG$ being strongly endotactic.  The proof that $\cG$ is strongly endotactic is identical except the line ``... is minimal over edges satisfying  $\bm{v}(\tilde{e}_j)\cdot \b{w} > 0$'' is changed to ``...  is minimal over all edges of $\tilde \cG$'' and by noting that the source complexes of $\cG$ are a subset of the source complexes of $\tilde \cG$.
\end{proof}

\subsection{Endotactic and Source-Only Networks}

We  introduce a new concept, that of  ``source-only'' networks, which we will demonstrate is a useful framework. 

\begin{definition}
A chemical reaction network $\cG = (\cV,\cE)$ is said to be \textbf{source-only} if, for any $e\in \cE$, $\bm{t}(e) = \b{y}$ implies that $\bm{s}(e^*) = \b{y}$ for some $e^* \in E$. A mass action system $\frac{d\mathbf{x}}{dt} = \b{f}(\mathbf{x})$ is said to be \textbf{source-only} if it can be generated by a source-only E-graph. A chemical reaction network $\cG$ is said to be \textbf{effectively source-only} if every system generated by $\cG$ is source-only.
\end{definition}
\noindent That is, a network is source-only if $\cV$ does not contain any nodes that are \textit{only} product nodes. 

\begin{example}
	\label{example1}
	Consider the chemical reaction network
	\begin{equation}
	\label{ex1}
	\xymatrix@R=1mm@C=3mm{
		3X_1 \ar[rr]^{k_1} & & 3X_2 \ar[ddl]^{k_2} & & & &  \\
		& & & & X_1 + X_2 \ar[rr]^{k_4}&& 2X_1 + 2X_2 .\\
		& \emptyset \ar[uul]_{k_3} & & & & &
	}
	\end{equation}
	The network can be defined by the E-graph in \cref{figure1}(a). We can see that the network is strongly endotactic. We next seek to represent the network as a source-only network. To do so we must dispose of the product complex $2X_1 + 2X_2$. We see that the fourth reaction can be split (in the sense of \cref{splitting}) to give the following dynamically equivalent reaction network
	\begin{equation}
	\label{ex12}
	\xymatrix@R=3mm@C=3mm{
		&  \ar[dl]_{k_4} X_1 + X_2 \ar[dr]^{k_4} & \\
		3X_1 \ar[rr]^{k_1} & & 3X_2 \ar[dl]^{k_2}\\
		& \ar[ul]^{k_3} \emptyset &
	}
	\end{equation}
	This network can be defined by the E-graph in \cref{figure1}(b).   Thus we see that the network \eqref{ex1} is effectively source-only.  \hfill $\triangle$
	\begin{figure}
	\centering
\includegraphics[scale =1]{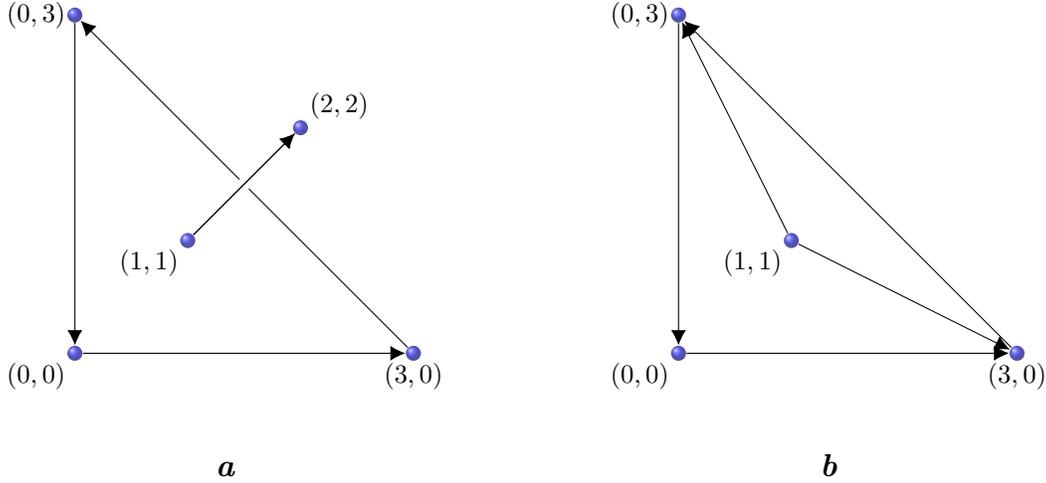}
	\caption{Two E-graphs $\cG$ (a) and $\tilde \cG$ (b) which have the capability of generating the same mass-action system (i.e. $\cG  \sqcap \tilde \cG$). Notice that, while both networks are strongly endotactic, the network $\tilde \cG$ is source-only which the network $\cG$ is not. There is, however, no weakly reversible E-graph $\bar \cG$ such that $\cG \sqcap \bar \cG$.}\label{figure1}
	\end{figure}	
\end{example}

We now prove the following.

\begin{lemma}
\label{sourceonly}
Let $\cG = (\cV,\cE)$ be an endotactic network. Then, there is exists a source-only network $\tilde{\cG}= (\tilde{\cV},\tilde{\cE})$ such that $\cG \sqsubseteq \tilde{\cG}$, and the nodes of $\tilde{\cG}$ are the source nodes of $\cG$. Therefore, every endotactic network is effectively source only.
\end{lemma}

\begin{proof}
We assume $\cG = (\cV,\cE)$ is endotactic and will construct a source-only network $\tilde{\cG}=(\tilde{\cV},\tilde{\cE})$ such that $\cG \sqsubseteq \tilde{\cG}$. 

We take the nodes of our new network to be the source complexes of the original network.  That is,  $\tilde \cV = \cS\cC_{\cG}$.  Next, we  define $\tilde \cE$ in the following way. First, we index the edges of $\cE$ as $e_1,...,e_k$. For each edge $e_i \in \cE$, if $\b{t}(e_i) \in \cS\cC_{\cG}$, we include $e_i\in \tilde{\cE}$. Otherwise, let $C_i^*$ be the set of edges of the complete graph on $\cS\cC_{\cG}$ with source $\b{s}(e_i)$, so that $\{\b{v}(e^*)|e^*\in C_i^*\} = \{\b{s}^*-\b{s}(e_i)|\b{s}^* \in (\cS\cC_{\cG} \setminus \{\b{s}(e_i)\})\}$. 

We next show that $\b{v}(e_i) \in \mathit{Cone}(\{\b{v}(e^*)|e^*\in C_i^*\}) = \mathit{Cone}(\{\b{s}^*-\b{s}(e_i)|\b{s}^* \in (\cS\cC_{\cG} \setminus \{\b{s}(e_i)\})\})$. If this was not the case, then by \cref{lemma01} there exists some $\b{w}\in \bR^d$ such that $\b{w}\cdot \b{v}(e_i) < 0$ and $\b{w}\cdot (\b{s}^* - \b{s}(e_i)) \geq 0$ for all $\b{s}^*\in \cS\cC_{\cG}$. However, if there exists any $\b{w}\in \bR^d$ such that $\b{w}\cdot \b{v}(e_i) < 0$, the condition that $\cG$ is endotactic implies that for some source $\b{s}^* \in \cS\cC_{\cG}$, and therefore some $e^*\in C_i^*$ with $\b{v}(e^*) = \b{s}^*-\b{s}(e_i)$, $\b{w}\cdot \b{v}(e^*) < 0$, because $\cS\cC_{\cG} = \{\b{s}(e^*)|e^*\in C_i^*\} \cup \{\b{s}(e_i)\}$.

Let $C_i\subset C_i^*$ be the set so that $\mathit{Cone}(\{\b{v}(e)|e\in C_i\})$ is minimal over subsets of $C_i^*$ in the sense of inclusion while still satisfying $\b{v}(e_i) \in \mathit{Cone}(\{\b{v}(e)|e\in C_i\})$. Therefore, $\b{v}(e_i) \in \mathit{RelInt}(\mathit{Cone}(\{\b{v}(e)|e\in C_i\}))$. Then, we add the edges $\tilde{e}\in C_i$ to $\tilde{\cE}$. By this construction, we have for each $\b{s} \in \cS\cC_{\cG}$, 
\[
\mathit{RelInt}(V^{\cG}(\bm{s})) \subseteq \mathit{RelInt}(V^{\tilde \cG}(\bm{s}))
\]
and we can conclude that $\cG \sqsubseteq \tilde{\cG}$, where $\tilde{\cG}$ is clearly source only.
\end{proof}

Notice that while \cref{sourceonly} guarantees that any endotactic dynamical system can be generated by source-only network, this source only network is not necessarily endotactic. The following example shows that we cannot guarantee that any endotactic dynamical system can be generated by a source-only endotactic network. 

\begin{example}

In \cref{figure:endocounter} (a), we give an example of a strongly endotactic network $\cG$ such that if $\tilde{\cG}$ is source only and $\cG \sqsubseteq \tilde{\cG}$, then $\tilde{\cG}$ is not endotactic. Splitting the edge labeled $e$ in the sense of \cref{splitting} requires adding a new edge $e^1$   such that $\b{v}(e^1) \cdot (1,0) <0$ and an edge $e^2$ such that $\b{v}(e^2) \cdot (0,-1) <0$. Then, $e^1$ can be used to show that the resulting network is not endotactic. Therefore, we cannot split any edges to maintain the endotactic property. It follows that to make the network source-only and endotactic, we must add a source node.

Consider again the node labeled $e$ in \cref{figure:endocounter} (a). To make the network source-only without splitting any edges, we must add source node at some point $\b{s}^* = \b{s}(e) + \alpha \b{v}(e)$ for $\alpha >0$. To maintain the endotactic property and insure that the new network $\tilde{\cG}$ contains the dynamics of $\cG$, we must add edges with source $\b{s}^*$ such that $\b{0} \in V^{\tilde{G}}(\b{s}^*)$. However, this requires either adding a new node which is only a target or adding a vector $e^1$ such that $\b{v}(e^1) \cdot (1,0) <0$ which can be used to show the resulting network is not endotactic. To make this target into a source, we have the same requirements. We conclude that there is no way to construct endotactic and source-only $\tilde{\cG}$ such that $\cG \sqsubseteq \tilde{\cG}$. 
\begin{figure}
\centering
\includegraphics[scale = 1]{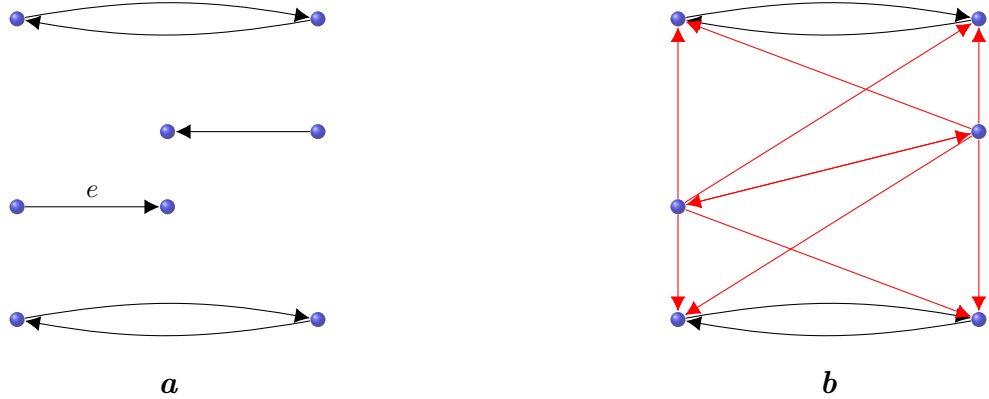}
\caption{(a) A strongly endotactic network $\cG$ such that if $\tilde{\cG}$ is source only and $\cG \sqsubseteq \tilde{\cG}$, then $\tilde{\cG}$ is not endotactic. (b) An example of a network $\tilde{\cG}$ constructed as in the proof of \cref{sourceonly} such that $\tilde{\cG}$ is source only and $\cG \sqsubseteq \tilde{\cG}$.}
\label{figure:endocounter}
\end{figure}
\hfill $\triangle$
\end{example}

Notice that the set of monomials $\{\mathbf{x}^{\bm{s}_i}\}$ of a generated polynomial $\bm{f}_{\cG(\cK)}$ corresponds to a subset of the source complexes $\cS\cC_{\cG}$, and including (but not limited to) the set of sources $s\in \cS\cC_{\cG}$ that have $\b{0} \not\in \mathit{RelInt}(V^{\cG}(s))$. To limit networks that must be considered given a polynomial, we wish to exclude sources from a network which do not appear as monomials. For the case of weakly reversible systems, the following theorem allows us to do this (see also Theorem 4.8 in \citep{craciun2018}).

\begin{theorem}\label{noaddedwkrev}
	If a polynomial dynamical system $\frac{d\mathbf{x}}{dt} = f(\mathbf{x})$ is weakly reversible, then it is generated by a weakly reversible network that has as its sources the exponent vectors of the monomials of $f$.
\end{theorem} 

\begin{proof}
	Let $\cG = (\cV,\cE)$ be a weakly reversible network which generates $f(\mathbf{x})$. If there is a node $\b{s}^* = \bm{s}(e)$, $e \in \cE$ such that the monomial $\mathbf{x}^{\b{s}^*}$ does not appear in $f$, we first introduce a term $0\mathbf{x}^{\b{s}^*}$ in $f$. It is now convenient to order such nodes $\b{s}^*_1,\b{s}^*_2,...,\b{s}^*_k$. Consider first $\b{s}^*_1$. We index the edges $e\in \cE$ such that $\b{s}(e) = \b{s}^*_1$ as $e^1_1,...,e^1_p$. Because the coefficient of the monomial $\mathbf{x}^{\b{s}^*_1}$ in $f$ is $0$, the vectors $\b{v}(e)$ and rate constants $k_e$ must satisfy $ \sum_{i=1}^p  k_{e_i^1} \bm{v}(e_i^1) = \b{0}$. For any edge $e^{\dagger}$ which has $\bm{t}(e^{\dagger}) = \b{s}^*_1$, we have in $f$ a term of the form $k^{\dagger} \mathbf{x}^{\bm{s}(e^{\dagger})}\bm{v}(e^{\dagger})$. We let
	\[
	\tilde{k}_{e_i^1}= \frac{k_{e_i^1}}{ \sum_{j=1}^p k_{e_j^1}} k^{\dagger}
	\]
	so that $k^{\dagger} = \sum_{i=1}^p \tilde{k}_{e_i^1}$, 
	\[
	\sum_{i=1}^p \tilde{k}_{e_i^1} \bm{v}(e_i^1) = \frac{k^{\dagger}}{\sum_{i=1}^p k_{e_i^1}}\sum_{i=1}^p k_{e_i^1}\b{v}(e_i^1) = 0
	\]
	and finally
	\[
	k^{\dagger}\bm{v}(e^{\dagger}) = \sum_{i=1}^p \tilde{k}_{e_i^1} \bm{v}(e^{\dagger}).
	\]
	Then, we can write
	\[
	k^{\dagger}\bm{v}(e^{\dagger}) = \sum_{i=1}^p \tilde{k}_{e_i^1} \bm{v}(e^{\dagger})+ \sum_{i=1}^p \tilde{k}_{e_i^1} \bm{v}(e_i^1)= \sum_{i=1}^p\tilde{k}_{e_i^1}  (\bm{v}(e^{\dagger}) + \bm{v}(e_i^1)).
	\]
	Therefore, we can replace the term $k^{\dagger} \mathbf{x}^{\bm{s}(e^{\dagger})}\bm{v}(e^{\dagger})$ in $f$ with 
	\[
	  \sum_{i=1}^p\tilde{k}_{e_i^1}  (\bm{v}(e^{\dagger}) + \bm{v}(e_i^1))\mathbf{x}^{\bm{s}(e^{\dagger})}.
	\]
	This implies that $f$ is generated by a network which is built from $\cG$ in the following way:
	\begin{enumerate}
		\item All of the edges $e$ which have $\bm{s}(e) = \b{s}^*$ or $\bm{t}(e) = \b{s}^*$ are removed. Let $E_1$ be the edges with $\bm{s}(e) = \b{s}^*$ and let $E_2$ be the edges with $\bm{t}(e) = \b{s}^*$.
		\item Edges are added from each source node of an edge in $E_2$ to each target node of an edge in $E_1$.
	\end{enumerate}
	The resulting network is weakly reversible because $\cG$ was weakly reversible, and the only paths removed consisted of an edge in $E_2$ followed by an edge in $E_1$. These were then replaced with a single edge (or no edge if the path went from a node to itself). We now have a network which is weakly reversible and generates $f(\mathbf{x})$, but does not include $\b{s}_1^*$ as a source. We then simply repeat the argument for $\b{s}_2^*,...,\b{s}_k^*$ to eliminate all nodes that do not appear as non-zero terms in $f(\mathbf{x})$.
\end{proof}

The class of source-only networks is useful because they provide an upper bound on the networks we must consider when we attempt to represent polynomial dynamical systems using various kinds of networks. Furthermore, \cref{noaddedwkrev} shows that in some cases one need only consider networks without adding new nodes. Therefore, source-only representations of networks provide \emph{finite} descriptions of reversible, weakly reversible, and endotactic networks, which is useful in computations. For example, in \citep{craciun2018}, it is shown that to find a complex balanced realization of a polynomial system, one need only consider the complexes that appear as exponent vectors. Furthermore, knowing that one can write a network as a source-only network is important in dynamical equivalence and network translation-based computational methods. In these settings, it is often required to know the number and/or stoichiometry of required complexes \citep{J1,johnston2013computing}.

\subsection{Endotactic and Consistent Networks}

We continue with results related to endotactic networks.

\begin{lemma}\label{consis}
Every endotactic network is consistent.
\end{lemma}

\begin{proof}
Suppose, in order to find a contradiction, that there is a network $(\cV,\cE)$ that is endotactic but not consistent. Since the network is not consistent, there does not exist a set of  constants $\lambda_e >  0$ for which 
\[
\sum_{e\in \cE} \lambda_e \bm{v}(e) = \mathbf{0}.
\]
It follows that condition 1. of \cref{stiemke} is not  satisfied,    so that condition 2. must be satisfied. That is, there is a $\b{w} \in \mathbb{R}^d$   such that $\b{w} \cdot \bm{v}(e) \leq 0$, with at least one inequality strict. It follows immediately from the definition of endotactic in  \cref{classificationsEG} that the network is not endotactic, which is a contradiction. It follows that every endotactic network is consistent.
\end{proof}

\subsection{Endotactic and Weakly Reversible Networks}

We have seen that every endotactic network may be represented in a dynamically equivalent form as a source-only network. Since weakly reversible networks are source-only by definition, it is tempting to suppose that every endotactic network is effectively weakly reversible. However, this is not the case, as we will show. In this section, we introduce the concept of \emph{extremal reactions} which helps to bridge the gap between endotactic and weakly reversible networks.

Considering \cref{example1} again, we see that this network is endotactic and effectively source-only, but that it is not weakly reversible. In order to make it weakly reversible, we must be able to reconfigure the other reactions so that $X_1 + X_2$ is also the product of some reaction. We observe, however, that we cannot ``split'' any of the other three reactions as they lie on the outer hull of the source complexes. For instance, to split $3X_1 \to 3X_2$ to connect to $X_1 + X_2$, we must necessarily introduce a balancing reaction which points away from the convex hull of the source complexes, and therefore introduces a strictly product complex. This network, therefore, is not effectively weakly reversible.

We can quickly identify that the reason there is no weakly reversible dynamically equivalent network is that there is a complex in the interior of the convex hull of the complexes which cannot be reached. In this example, however, we might observe that the restriction of the network to just the boundary complexes $3X_1$, $3X_2$, and $\emptyset$ \emph{is} weakly reversible. This is perhaps not surprising; after all, we observed that the reactions which could not be ``split'' were exactly those which were on the boundary of the convex hull. We therefore introduce the following.

\begin{definition}
	Consider a chemical reaction network $\cG = (\cV,\cE)$.
	We define the \textbf{extreme source complexes} $\cE\cC_{\cG} \subseteq \cS\cC_{\cG} $ to be  the set of source nodes $\bm{s}(e)$ which are on the border of the convex hull of $\cS\cC_{\cG}$. The \textbf{extremal reaction set} $\cE\cE_{\cG}$ is defined to be the set $\{e\in \cE| \bm{s}(e) \in \cE\cC_{\cG} \}$, and we define $\cE\cV_{\cG}$ to be the subset of $\cV$ that are the sources and/or targets of the edges in $\cE\cE_{\cG}$.
	Then the network  $\cG$ is said to be \textbf{extremally weakly reversible} if the reduced network $(\cE\cV_{\cG},\cE\cE_{\cG} )$ is weakly reversible. 
	
	A mass action system is said to be \textbf{extremally weakly reversible} if it is generated by an extremally weakly reversible E-graph. A chemical reaction network $\cG$ is said to be \textbf{effectively extremally weakly reversible} if any system generated by $\cG$ is extremally weakly reversible. 
\end{definition}

It is clear that for \cref{example1} we have that $\mathcal{E}\cC_{\cG} = \cE\cV_{\cG} = \{ 3X_1, 3X_2, \emptyset \}$ and that $(\mathcal{E}\cV_{\cG} ,\cE\cE_{\cG})$ is  weakly reversible. Hence, the original network is extremally weakly reversible.    

 We wish to determine how robust this property is among endotactic networks. Consider the following example.

\begin{example}
\label{example2}
Consider the network
\begin{equation}
\label{ex3}
\xymatrix@R=5mm@C=7mm{
	X_2 \ar[r]^{k_1} & X_1 + 3X_2 \ar@<0.5ex>[r]^{k_2} & \ar@<0.5ex>[l]^{k_3} 2X_1 + 3X_2\\
	3X_1 + 2X_2 \ar[r]^{k_4} &2X_1 \ar@<0.5ex>[r]^{k_5} & \ar@<0.5ex>[l]^{k_6} X_1
}
\end{equation}
The network can also be represented by the E-graph in \cref{figure2}. It can be visually checked that the network is endotactic. Furthermore, we can see that every complex is an extremal complex, so that $\mathcal{E} \cC_{\cG}= \cS\mathcal{C}_{\cG}$, and $\cE\cE_{\cG} = \cE$. Nevertheless, no reaction may be ``split'' while preserving the property that the network is source-only. It follows that the network is neither effectively weakly reversible nor effectively extremally weakly reversible.  \hfill $\triangle$ 
\begin{figure}[h]
\centering
\includegraphics[scale = 1]{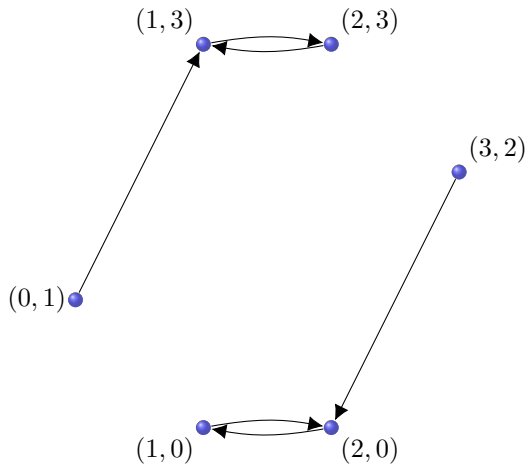}
\caption{E-graph defining the network \cref{ex3}. Although the network is source-only and endotactic, it does not permit a strongly endotactic, weakly reversible, or extremally weakly reversible representation.}
\label{figure2}
\end{figure}
\end{example}

While the network in \cref{example2} is endotactic, it is not strongly endotactic. We now consider whether the property of being effectively extremally weakly reversible holds for strongly endotactic networks. Consider the following example.

\begin{example}
\label{example3}
Consider the network given in the three-dimensional complex space by \cref{figure3}. It can be verified by visual inspection that the network is strongly endotactic. We also have that  $\cE\cC_{\cG} = \cS\cE_{\cG}$; however we again may not ``split'' any reaction from these complexes while maintaining the property of being source-only. It follows that the network is not effectively weakly reversible.  \hfill $\triangle$
\begin{figure}[h]
\centering
\includegraphics[scale = 1]{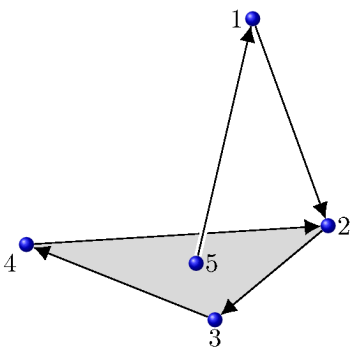}
\caption{Three-dimensional E-graph defining a network. Although the network is source-only and strongly endotactic, it does not permit a weakly reversible or extremally weakly reversible representation.}
\label{figure3}
\end{figure}
\end{example}

This example is three-dimensional. The following result considers  strongly endotactic networks which have a two-dimensional stoichiometric subspace. 

\begin{theorem}\label{wkrev}
	Let  $\cG = (\cV,\cE)$ be a strongly endotactic two-dimensional network with two-dimensional stoichiometric subspace and assume that the source complexes only reside  on the boundary of the convex hull generated from the source complexes.  Then there exists a weakly reversible Euclidean embedded graph $\tilde{\cG}$ such that $\cG \sqsubseteq \tilde{\cG}$. Therefore, every two dimensional strongly endotactic network is effectively extremally weakly reversible.
\end{theorem}

\begin{figure}
\centering
\includegraphics[scale = 0.6]{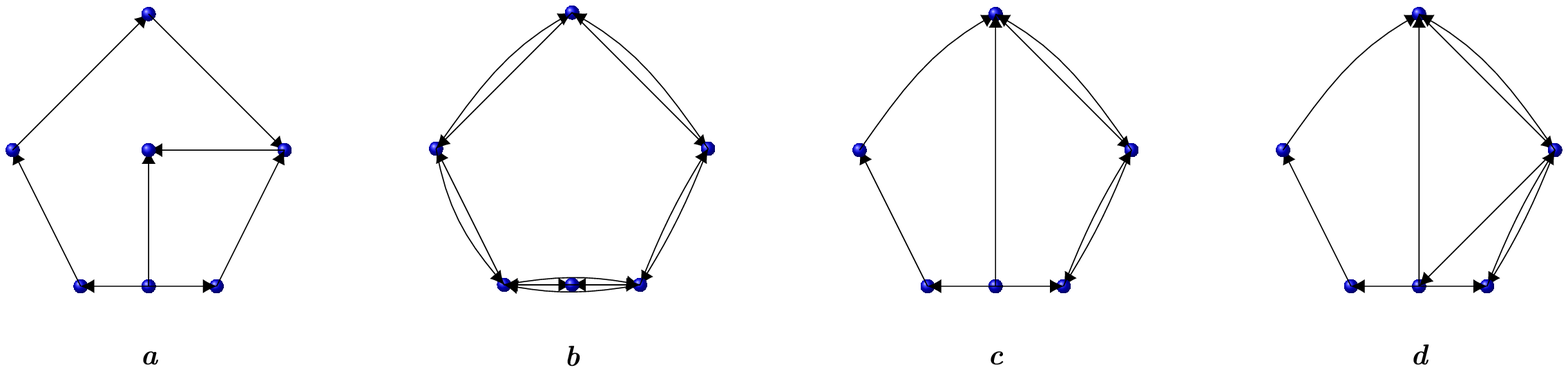}
	\caption{The networks constructed as outlined by the proof of \cref{wkrev}.  (a) The initial E-graph $\cG$. (b) $\cG_1$, which includes all of the possible edges that lie along a face of the convex hull of $\cS\cC_{\cG}$. (c) $\cG_2$, which is strongly endotactic and has $\cG\sqsubseteq \cG_2$. (d) $\tilde{\cG}$, which is strongly endotactic, weakly revesible, and has $\cG \sqsubseteq \tilde{\cG}$.}
\end{figure}

\begin{proof}
Let $\mathcal{G} = (\cV,\cE)$ be a strongly endotactic network with sources only on the convex hull of source complexes. We will build $\tilde{\cG}$ in three stages, constructing ``intermediate" networks $\cG_1$ and $\cG_2$, and finally $\tilde{\cG}$.

Consider the Euclidean embedded graph $\cG_1=(\cV_1,\cE_1)$ with $\cV_1= \cS\cC_{\cG}$, and as its edges $\cE_1$ all possible edges which lie along the sides of the convex hull of $\cS\cC_{\cG}$. $\cG_1$ is clearly weakly reversible, and while $\cS\cC_{\cG_1} = \cS\cC_{\cG}$, $\cG_1$ does not necessarily contain the dynamics of $\cG$. 

We now prove, however, that for each $\b{s}\in \cS\cC_{\cG}$, we have that either $V^{\cG}(\bm{s}_i) \subseteq V^{\cG_1}(\bm{s}_i)$ or $V^{\cG_1}(\bm{s}_i)$ is the border of a half space which contains $V^{\cG}(\bm{s}_i)$. Let $\cW_i$ be the set such that for $\b{w} \in \cW_i$, $\b{w} \cdot (\b{s}_i-\b{s}^*) \leq 0$ for any $\b{s}^* \in \cS\cC_{\cG}$. This set contains some non-zero vector because $\b{s}_i$ is in the convex hull of $\cS\cC_{\cG}$. Let $\b{w}\in \cW$ and $e\in \cE$ be such that $\b{s}(e) = \b{s}_i$. Because $\cG$ is endotactic and $\b{w} \cdot (\b{s}_i-\b{s}^*) \leq 0$, it must be that $\b{v}(e) \cdot \b{w} \geq 0$. If $V^{\cG_1}(\bm{s}_i)$ is pointed, then $\cW_i$ is the dual cone to $V^{\cG_1}(\bm{s}_i)$ and by \cref{lemma01} we can conclude that $V^{\cG}(\bm{s}_i) \subseteq V^{\cG_1}(\bm{s}_i)$. If $V^{\cG_1}(\bm{s}_i)$ is a line (note that by construction $V^{\cG_1}(\bm{s}_i)$ cannot be a ray), then $\cW_i$ is perpendicular to $V^{\cG_1}(\bm{s}_i)$. Then, we can conclude that $V^{\cG_1}(\bm{s}_i)$ is the border of a half space which contains $V^{\cG}(\bm{s}_i)$, and that this half space also contains $\cS\cC_{\cG}$, because $\b{w} \cdot (\b{s}-\b{s}^*) \leq 0$ for any $\b{s}^* \in \cS\cC_{\cG}$.

\medskip

We next construct $\cG_2 = (\cV_2,\cE_2)$ from $\cG_1$ such that $\cG_2$ contains the dynamics of $\cG$ by modifying the edge set $\cE_1$. We index the set $\cS\cC_{\cG}$ of source complexes as $\b{s}_1,...,\b{s}_{|\cS\cC_{\cG}|}$ and build $\cE_2$ by adding edges $e$ with $\b{s}(e) = \b{s}_i$ for $i=1,...,|\cS\cC_{\cG}|$. For each source $\bm{s}_i$ of $\cG_1$, there are three possibilities we must consider:
\begin{enumerate}[{\bf(a)}]

\item If $\mathit{RelInt}(V^{\cG}(\bm{s}_i)) \subseteq \mathit{RelInt}(V^{\cG_1}(\bm{s}_i))$, we simply include the edges $e$ of $\cG_1$ with $\b{s}(e) = \b{s}_i$ in $\cE_2$. Then, clearly $\mathit{RelInt}(V^{\cG}(\bm{s}_i)) \subseteq \mathit{RelInt}(V^{\cG_2}(\bm{s}_i))$. 

\item If $V^{\cG_1}(\bm{s}_i)$ is a line and $V^{\cG}(\bm{s}_i)$ is not contained in that line, we again include each edge $e$ of $\cG_1$ with $\b{s}(e) = \b{s}_i$ in $\cE_2$, but must also include additional edges. We know that, in this case, $V^{\cG_1}(\bm{s}_i)$ is the border of a half space which contains $V^{\cG}(\bm{s}_i)$ and $\cS\cC_{\cG}$. We add an edge with source $\b{s}_i$ and a target in $\cV_1= \cS\cC_{\cG}$ that does not lie in the line $V^{\cG_1}(\bm{s}_i)$. Thus, $V^{\cG_2}(\bm{s}_i)$ is the appropriate half space. Then, $\mathit{RelInt}(V^{\cG}(\b{s}_i)) \subseteq \mathit{RelInt}(V^{\cG_2}(\b{s}_i))$.

\item If $V^{\cG}(\bm{s}_i) \subset \partial V^{\cG_1}(\bm{s}_i)$, then $V^{\cG}(\bm{s}_i)$ is a ray along one face of the convex hull of $\cS\cC_{\cG}$. We then take only the edges of $\cG_1$ which lie along this ray to be edges in $\cE_2$. Then, clearly $\mathit{RelInt}(V^{\cG}(\bm{s}_i)) \subseteq \mathit{RelInt}(V^{\cG_2}(\bm{s}_i))$ because $V^{\cG}(\bm{s}_i) = V^{\cG_2}(\bm{s}_i)$. 
\end{enumerate}

We have now constructed $\cG_2 = (\cV_2,\cE_2)$ such that $\cG \sqsubseteq \cG_2$. Furthermore, $\cG_2$ is strongly endotactic, as we now show. Let $\b{w} \in \bR^d$ and $e \in \cE_2$ be such that $\b{w}\cdot \b{v}(e) < 0$. Let $\{\b{s}^*\}\subset \cE_2$ be the sources such that $\b{w}\cdot (\b{s}^* - \b{s}(e_j)) \leq 0$ for all $e_j \in \cE_2$, and $\{e^*\}$ the corresponding edges. Any edges in $\cG_2$ correspond to reaction vectors which do not point out of the convex hull of $\cS\cC_{\cG}$, so we know that $\b{w}\cdot(\b{s}^* - \b{s}(e)) <0$. Also, $\{\b{s}^*\}$ is the same as the set of sources of $\cG$ which have $\b{w}\cdot (\b{s}(e_j) - \b{s}^*) \leq 0$ for all $e_j \in \cE$. If none of the $\b{v}(e^*)\cdot \b{w} >0$, then our construction implies this is true of the edges of $\cG$ as well. This contradicts the assumption that $\cG$ is strongly endotactic. We conclude that $\cG_2$ is strongly endotactic.

\medskip

It is possible that $\cG_2$ is not weakly reversible, so we finally construct $\tilde{\cG}$ such that $\tilde{\cV} = \cV_2$, $\cE_2 \subseteq \tilde{\cE}$, and both $\tilde{\cG} \sqsubseteq \cG_2$ and $\cG_2 \sqsubseteq \tilde{\cG}$ hold. To complete the construction, we must first establish the following about the structure of $\cG_2$:
\begin{enumerate}[\bf (i)]
\item For each $\b{s}\in \cS\cC_{\cG}$, either $V^{\cG_2}(\b{s})$ is one dimensional and intersects a (one dimensional) face of the convex hull of $\cS\cC_{\cG}$, or $V^{\cG_2}(\b{s})$ is solid (meaning it has two-dimensional span) and $(\b{s}_i-\b{s}) \in V^{\cG_2}(\b{s})$ for all $\b{s}_i \in \cS\cC_{\cG}$.
\item On every face of the convex hull of $\cS\cC_{\cG}$, at least one of the following is true: there is some $\b{s}$ such that $(\b{s}_i-\b{s}) \in V^{\cG_2}(\b{s})$ for all $\b{s}_i \in \cS\cC_{\cG}$ (and $V^{\cG_2}(\b{s})$ is solid), or there is some $\b{s}$ such that $\b{s}$ is a corner of the convex hull of $\cS\cC_{\cG}$ and $V^{\cG_2}(\b{s})$ is a ray pointing along an adjacent face of the convex hull of $\cS\cC_{\cG}$. 
\item We may add a path from any source on some face of the convex hull of $\cS\cC_{\cG}$ to a source as in (ii) to create a network $\cG^*$ such that $\cG^*\sqsubseteq \cG_2$ and $\cG_2 \sqsubseteq \cG^*$.
\end{enumerate}

To establish (i), suppose that we have some $\b{s}$ and $V^{\cG_2}(\b{s})$ is not one dimensional. Then, either $V^{\cG_2}(\b{s}) = V^{\cG_1}(\b{s})$ or $V^{\cG_2}(\b{s})$ is a half space such that $V^{\cG_1}(\b{s}) = \partial V^{\cG_2}(\b{s})$. In either case, the convexity of the convex hull of $\cS\cC_{\cG}$ implies (i). By our construction possibility (c), if $V^{\cG_2}(\b{s})$ is one dimensional, it must intersect a (one dimensional) face of the convex hull of $\cS\cC_{\cG}$.

To establish (ii), let $\b{w}$ be such there is some set $S \subset \cS\cC_{\cG}$ with at least two distinct elements and $\b{w} \cdot (\b{s}_i-\b{s}_j) = 0$ for $\b{s}_i,\b{s}_j \in S$ and $\b{w} \cdot (\b{s}_i-\b{s}_k) < 0$ for $\b{s}_i \in S$, $\b{s}_k \not\in S$ (i.e., $\b{w}$ is the inward pointing normal to a one dimensional face of the convex hull of $\cS\cC_{\cG}$). $\cG_2$ is strongly endotactic, so for some $\b{s}\in S$, there is some $e\in \cE_2$ with $\b{s}(e) = \b{s}$ and $\b{v}(e) \cdot \b{w} >0$. We can conclude using fact (a) that on every face of the convex hull of $\cS\cC_{\cG}$, at least one of the following is true: there is some $\b{s}$ such that $(\b{s}_i-\b{s}) \in V^{\cG_2}(\b{s})$ for all $\b{s}_i \in \cS\cC_{\cG}$ (and $V^{\cG_2}(\b{s})$ is solid), or there is some $\b{s}$ such that $\b{s}$ is a corner of the convex hull of $\cS\cC_{\cG}$ and $V^{\cG_2}(\b{s})$ is a ray pointing along an adjacent face of the convex hull of $\cS\cC_{\cG}$. 

To establish (iii), let $\cS$ be a face of the convex hull of $\cS\cC_{\cG}$, and let $\b{s} \in \cS$ be a source such that either $(\b{s}_i-\b{s}) \in V^{\cG_2}(\b{s})$ for all $\b{s}_i \in \cS\cC$ or $V^{\cG_2}(\b{s})$ is a ray pointing along an adjacent face of the convex hull of $\cS\cC_{\cG}$. Let $\b{s}^* \in \cS$ be some other source on the same face of the convex hull of $\cS\cC_{\cG}$. If $V^{\cG_2}(\b{s}^*)$ is solid or a full line, then $\b{s}-\b{s}^*\in V^{\cG_2}(\b{s}^*)$, and so if $\cG^*$ is the network with an edge added from $\b{s}^*$ to $\b{s}$, then $V^{\cG^*}(\b{s}^*) = V^{\cG_2}(\b{s}^*)$. If $V^{\cG_2}(\b{s}^*)$ is a ray and $\b{s}-\b{s}^* \not \in V^{\cG_2}(\b{s}^*)$, there must be some $\b{s}^{**}$ such that $(\b{s}^{**}-\b{s}^*) \in V^{\cG_2}(\b{s}^*)$ and either $\b{s}^{**}$ also has either $(\b{s}_i-\b{s}^{**}) \in V^{\cG_2}(\b{s}^{**})$ for all $\b{s}_i \in \cS\cC$ or $V^{\cG_2}(\b{s}^{**})$ is a ray pointing along an adjacent face of the convex hull of $\cS\cC_{\cG}$, or $V^{\cG_2}(\b{s}^{**}) = - V^{\cG_2}(\b{s}^{*})$. In the last case, $(\b{s}-\b{s}^{**})\in V^{\cG_2}(\b{s}^{**})$, and so if $\cG^*$ is the network to which we added edges to form a path from $\b{s}^*$ to $\b{s}$ through $\b{s}^{**}$, we have that $V^{\cG^*}(\b{s}^*) = V^{\cG_2}(\b{s}^*)$ and $V^{\cG^*}(\b{s}^{**}) = V^{\cG_2}(\b{s}^{**})$ . 

The above arguments show that for any source $\b{s}^* \in \cS\cC_{\cG}$, there is a source $\b{s}$ in the same face of the convex hull of $\cS\cC_{\cG}$ such that either $(\b{s}_i-\b{s}) \in V^{\cG_2}(\b{s})$ for all $\b{s}_i \in \cS\cC$ or $V^{\cG_2}(\b{s})$ is a ray pointing along an adjacent face of the convex hull of $\cS\cC_{\cG}$, and furthermore that if $\cG^*$ is the network to which we added edges to form a path from $\b{s}^*$ to $\b{s}$, then $\cG^*\sqsubseteq \cG_2$ and $\cG_2 \sqsubseteq \cG^*$.

We may now complete the construction of $\tilde{\cG}$. Recalling that $\cG_1$ is weakly reversible, $\tilde{\cG}$ is weakly reversible if $\tilde{\cE}$ includes edges which replace any paths present in $\cG_1$ that were not included in $\cG_2$. Let $e_j$ be any edge in $\cE_1$ but not in $\cE_2$. Note that $\b{t}(e_j) = \b{s}(e_k)$ for some $e_k \in \tilde{\cE}$ (because $\tilde{\cV} = \cV_2=\cV_1 = \cS\cC_{\cG}$). We must add a path of edges in $\tilde{\cE}$ from $\b{s}(e_j)$ to $\b{s}(e_k)$, or prove that such a path is already present in $\tilde{\cE}$. Let $\b{s}(e_j) = \b{s}$. We have seen that we may add a path of edges to some source $\b{s}^*$ in the same of face of the convex hull of $\cS\cC_{\cG}$ as $\b{s}$ such that one of the following is true:
\begin{enumerate}[\bf(a)]
\item $(\b{s}_i-\b{s}^*) \in V^{\cG_2}(\b{s}^*)$ for all $\b{s}_i \in \cS\cC_{\cG}$, or
\item $\b{s}^*$ is is a corner of the convex hull of $\cS\cC_{\cG}$ and $V^{\cG_2}(\b{s}^*)$ is a ray pointing along an adjacent face of the convex hull of $\cS\cC_{\cG}$.
\end{enumerate}
If (a) holds, then we can add an edge to $\tilde{\cE}$ with source $\b{s}^*$ and target $\b{s}(e_k)$ and still have both $\tilde{\cG} \sqsubseteq \cG_2$ and $\cG_2 \sqsubseteq \tilde{\cG}$ , and a path from $\b{s}(e_j)$ to $\b{s}(e_k)$. 

If (b) holds but (a) does not, we repeat the argument on the adjacent face of the convex hull of sources for $\b{s}^*$, letting $\b{s}^{**}$ be the new source which satisfies one of (a) or (b). Note that $ V^{\cG_2}(\b{s}^*)$ is not solid, so $\b{s}^{**} \neq \b{s}^*$ (otherwise (a) was originally satisfied). If again only (b) holds, we may continue the argument until (a) holds for some $\b{s}$ in a face of the convex hull of $\cS\cC_{\cG}$, or (b) holds for some $\b{s}$ and $\b{s}(e_k)$ is in the face of the convex hull that $V^{\cG_2}(\b{s})$ points along.

We conclude that if $\tilde{\cG}$ is the weakly reversible network with paths added to replace any edges in $\cE_1$ that are missing from $\cE_2$, then $\tilde{\cG} \sqsubseteq \cG_2$. Then, \cref{effectively_endo_endo} implies that $\tilde{\cG}$ is strongly endotactic. Furthermore, $\cG\sqsubseteq \cG_2\sqsubseteq \tilde{\cG}$.
\end{proof}

\subsection{Additional Examples}

We now present counterexamples to various possible inclusions of network types in the sense of dynamical equivalence. This will allow us to conclude that any arrow added to our \cref{relationships}  would be false.
\begin{figure}[b!]
\centering
\includegraphics[scale = 0.7]{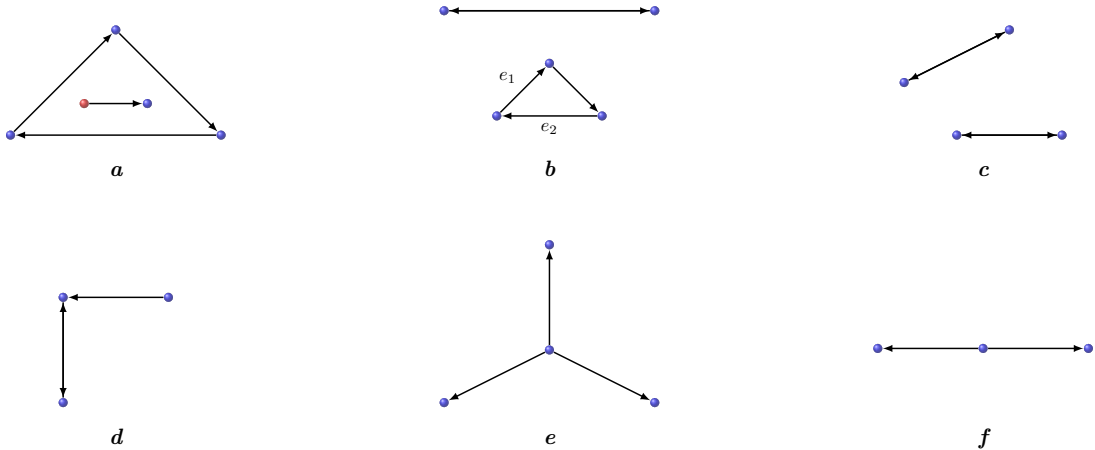}
\caption{Collection of examples. All examples are 2-dimensional networks, with the exception of (f), which is \cref{CnotSO}, that is 1-dimensional. (a) \cref{EWRnotWR}: Extremally weakly reversible but not effectively weakly reversible.(b) \cref{WRnotEWR}: Weakly reversible but not effectively extremally weakly reversible. (c) \cref{EWRnotSE}: Reversible but not effectively strongly endotactic. (d)\cref{SOnotE}: Source only but not effectively endotactic. (e) \cref{CnotE}: Consistent but not effectively endotactic. (f) \cref{CnotSO}: Consistent but not effectively source-only.
}
\label{counters}
\end{figure}

\begin{example}\label{EWRnotWR}
	Consider the E-graph shown in \cref{counters} (a). This network is extremally weakly reversible, but it is not effectively weakly reversible. According to \cref{noaddedwkrev}, if there is a weakly reversible network which generates a system generated by this network, it needs no added sources. Therefore, there must be a path from at least one extremal source to the interior source (shown in red). However, splitting (as in \cref{splitting}) any extremal reaction will result in a new reaction which points out of the convex hull of sources. The resulting network cannot be endotactic, and so is not weakly reversible. \hfill $\triangle$
\end{example}
\begin{example}\label{WRnotEWR}
	Consider the E-graph shown in \cref{counters} (b). This network is weakly reversible, while it is not effectively extremally weakly reversible. The extremal reaction set consists of two irreversible reactions which form a path (labeled $e_1$ and $e_2$) and one reversible reaction pair. Again, \cref{noaddedwkrev} implies that if there is a weakly reversible network generates a system generated by the extremal reaction set then it needs no added sources. While this does allow us to reverse edge $e_2$ in \cref{counters} (b) by splitting (as in \cref{splitting}) edge $e_1$, the result is a new irreversible path into the reversible reaction pair. Neither reversible reaction can be split without introducing a new reaction which points out of the convex hull of sources.  Thus, there is no weakly reversible network which contains the dynamics of this extremal reaction set. We conclude that the network is not effectively extremally weakly reversible. \hfill $\triangle$
\end{example}
\begin{example}\label{EWRnotSE}
	Consider the E-graph shown in \cref{counters} (c). This network is reversible, weakly reversible, and extremally weakly reversible, while it is not effectively strongly endotactic. No reaction present can be split (as in \cref{splitting}) while preserving the endotactic property. By \cref{containcondit}, any node added to create a new network $\tilde{\cG}$ must have $\b{0} \in \mathit{RelInt}(V^{\tilde{\cG}}(\b{s}))$. Therefore, any direction $\b{w}$ which violated the strongly endotactic conditions must still do so. \hfill $\triangle$
\end{example}

\begin{example}\label{SOnotE}
	Consider the E-graph shown in \cref{counters} (d). This network is source-only, while it is not effectively endotactic. No reaction present can be split (as in \cref{splitting}) to gain the endotactic property. By \cref{containcondit}, any node added to create a new network $\tilde{\cG}$ must have $\b{0}\in \mathit{RelInt}(V^{\tilde{\cG}}(\b{s}(e)))$. Therefore, any direction $\b{w}$ which violated the endotactic conditions must still do so. \hfill $\triangle$
\end{example}

\begin{example}\label{CnotE}
	Consider the E-graph shown in \cref{counters} (e). This network is consistent.  However, for  generic choices of rate constants, the polynomial dynamical systems generated by this network are also generated by a network with a single irreversible reaction.   Hence, the network is not effectively endotactic.  Note that the same is true for the E-graph shown in \cref{counters} (f). \hfill $\triangle$
\end{example}

\begin{example}\label{CnotSO}
	Consider the E-graph shown in \cref{counters} (f). This one-dimensional network is consistent, but it is not effectively source-only. Any system generated by this network has only a single term. Any other network $\tilde{\cG}$ which also generates such a system and contains more than one source must have additional sources which are extremal sources. However, \cref{containcondit} implies that these must have $\b{0}\in \mathit{RelInt}(V^{\tilde{\cG}}(\b{s}))$, and so must have target nodes outside of the convex hull of sources, which could therefore not be sources. \hfill $\triangle$
\end{example}

\section{Conclusion}

We have determined the extent to which different reaction networks may represent the same dynamical system when modeled with mass action kinetics. This allows us to investigate the overlap between classes of reaction networks, in the sense of dynamical equivalence and ``effective" properties. \cref{relationships} provides a summary of the relationships between classifications of networks, giving an answer to \cref{q2}. Furthermore, the graph in \cref{relationships} is complete in the sense that any additional arrows would be false, with the exception of arrows that are already implied by directed paths. 

Our answers to \cref{q1} and \cref{q2} provide a framework for the study of generic interaction networks, and indeed systems of ODEs with polynomial right hand sides, in the context of reaction network theory. Reaction network theory provides useful tools for the analysis of dynamical systems \citep{H-J1,C-F1,C-F2,W-H1,W-H2,E-T}, and an answer to \cref{q1} provides a way to extend these results to systems for which they are not immediately applicable. Our work on \cref{q2} organizes the hierarchy of the various results in reaction network theory, allowing them to be extended where appropriate.

\begin{figure}
\includegraphics[scale = 1]{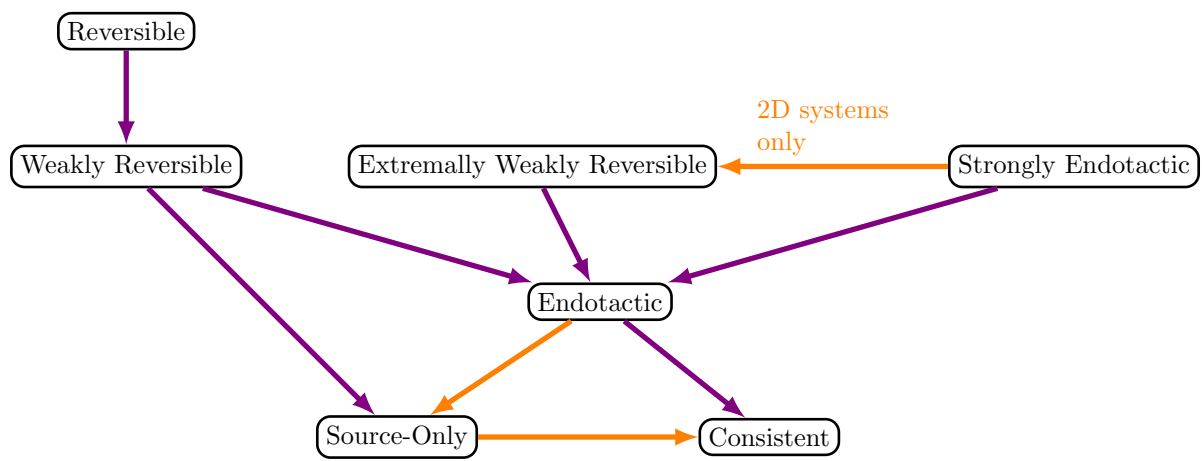}
\caption{This figure summarizes our main results on inclusions of classes of networks. Purple arrows indicate an inclusion in the family of networks, which also implies inclusion in the sense of dynamical equivalence. Network inclusions given by the purple arrows are strict at the level of networks and also at the level of dynamical equivalence.  Orange arrows indicate inclusion in the sense of dynamical equivalence only (or ``effective" properties). That is, an orange arrow indicates that a type of network at tail end is effectively the type of network at the head using \cref{sourceonly}, \cref{consis}, and \cref{wkrev}. The graph is complete in the sense that any additional paths would be false.}\label{relationships}
\end{figure}

\section{Acknowledgments}

David F. Anderson was supported by Army Research Office grant W911NF-18-1-0324. James D. Brunner was supported by the DeWitt \& Curtiss Family Foundation and the Mayo Clinic Center for Individualized Medicine. Gheorghe Craciun was partially supported by the National Science Foundation under grants DMS-1412643 and DMS-1816238. Matthew D. Johnston was supported by the Henry Woodward Fund.

\FloatBarrier

\bibliographystyle{plainnat}
\bibliography{myrefs}

\end{document}